\documentclass[11pt, reqno]{amsart}
\usepackage{amsfonts}
\usepackage{amsmath}
\usepackage{amssymb, color}
\usepackage{amscd}
\usepackage[mathscr]{eucal}
\usepackage{url}
\usepackage{enumerate}

\allowdisplaybreaks[1]

\renewcommand*{\bar}{\overline}

\theoremstyle{plain}

\newtheorem{theorem}{Theorem}[section]
\newtheorem{lemma}[theorem]{Lemma}
\newtheorem{prop}[theorem]{Proposition}
\newtheorem{cor}[theorem]{Corollary}
\theoremstyle{definition}

\numberwithin{equation}{section}

\makeatletter
\def\imod#1{\allowbreak\mkern5mu({\operator@font mod}\,\,#1)}
\makeatother

\makeatletter
\@namedef{subjclassname@2020}{%
  \textup{2020} Mathematics Subject Classification}
\makeatother

\begin{document}

\title[Generalized Paley Graphs]{Generalized Paley Graphs and their Complete Subgraphs of Orders Three and Four}

\author{Madeline Locus Dawsey, Dermot M\lowercase{c}Carthy}

\address{Madeline Locus Dawsey, Department of Mathematics, The University of Texas at Tyler, Tyler, TX 75799, USA}

\address{Dermot M\lowercase{c}Carthy, Department of Mathematics \& Statistics, Texas Tech University, Lubbock, TX 79410-1042, USA\\}

\email{mdawsey@uttyler.edu}

\email{dermot.mccarthy@ttu.edu}

\thanks{The second author was supported by a grant from the Simons Foundation (\#353329, Dermot McCarthy).}


\begin{abstract}
Let $k \geq 2$ be an integer. Let $q$ be a prime power such that $q \equiv 1 \imod {k}$ if $q$ is even, or, $q \equiv 1 \imod {2k}$ if $q$ is odd. The generalized Paley graph of order $q$, $G_k(q)$, is the graph with vertex set $\mathbb{F}_q$ where $ab$ is an edge if and only if ${a-b}$ is a $k$-th power residue. We provide a formula, in terms of finite field hypergeometric functions, for the number of complete subgraphs of order four contained in $G_k(q)$, $\mathcal{K}_4(G_k(q))$, which holds for all $k$. This generalizes the results of Evans, Pulham and Sheehan on the original ($k$=2) Paley graph. We also provide a formula, in terms of Jacobi sums, for the number of complete subgraphs of order three contained in $G_k(q)$, $\mathcal{K}_3(G_k(q))$. In both cases we give explicit determinations of these formulae for small $k$. We show that zero values of $\mathcal{K}_4(G_k(q))$ (resp.~$\mathcal{K}_3(G_k(q))$) yield lower bounds for the multicolor diagonal Ramsey numbers $R_k(4)=R(4,4,\cdots,4)$ (resp.~$R_k(3)$). We state explicitly these lower bounds for small $k$ and compare to known bounds. We also examine the relationship between both $\mathcal{K}_4(G_k(q))$ and $\mathcal{K}_3(G_k(q))$, when $q$ is prime, and Fourier coefficients of modular forms.
\end{abstract}

\maketitle


\section{Introduction}\label{sec_Intro}
It is well known that the two-color diagonal Ramsey number $R(4,4)$ equals $18$. This was first proved by Greenwood and Gleason \cite{GG} in 1955. They exhibited a self-complementary graph of order 17 which does not contain a complete subgraph of order four, thus showing $17<R(4,4)$, and then combined this with elementary upper bounds. The graph they describe is one in the family of graphs which are now known as Paley graphs. Let $\mathbb{F}_q$ denote the finite field with $q$ elements and let $S$ be its subset of non-zero squares. For $q\equiv 1 \pmod 4$ a prime power, the Paley graph of order $q$, $G(q)$, is the graph with vertex set $\mathbb{F}_q$ where $ab$ is an edge if and only if $a-b \in S$. The Paley graphs are connected, self complementary and strongly regular with parameters $(q, \frac{q-1}{2}, \frac{q-5}{4}, \frac{q-1}{4})$. Please see \cite{J} for more information on their main properties and for a nice exposition of their history since Paley's original paper \cite{P} in 1933. 

Let $\mathcal{K}_m(G)$ denote the number of complete subgraphs of order $m$ contained in a graph $G$. While the work of Greenwood and Gleason tells us that $\mathcal{K}_4(G(17)) =0$ and $\mathcal{K}_4(G(q))>0$ for $q>17$; Evans, Pulham and Sheehan \cite{EPS} provide a simple closed formula for $\mathcal{K}_4(G(p))$ for all primes $p \equiv 1 \pmod 4$. Write $p=x^2 +y^2$ for integers $x$ and $y$, with $y$ even. Then
\begin{equation}\label{for_EPS} 
\mathcal{K}_4(G(p)) = \frac{p(p-1)((p-9)^2-4y^2)}{2^9 \cdot 3}.
\end{equation}

In 2009, Lim and Praeger \cite{LP} introduced \emph{generalized Paley graphs}. Let $k \geq 2$ be an integer. Let $q$ be a prime power such that $q \equiv 1 \imod {k}$ if $q$ is even, or, $q \equiv 1 \imod {2k}$ if $q$ is odd. Let $S_k$ be the subgroup of the multiplicative group $\mathbb{F}_q^{\ast}$ of order $\frac{q-1}{k}$ containing the $k$-th power residues, i.e., if $\omega$ is a primitive element of $\mathbb{F}_q$, then $S_k = \langle \omega^k \rangle$. Then the \emph{generalized Paley graph} of order $q$, $G_k(q)$, is the graph with vertex set $\mathbb{F}_q$ where $ab$ is an edge if and only if $a-b \in S_k$. We note, due to the conditions imposed on $q$, that $-1 \in S_k$ so $G_k(q)$ is a well-defined undirected graph. $G_k(q)$ is connected if and only if $S_k$ generates $\mathbb{F}_q$ under addition.
When $k=2$ we recover the original Paley graph.

The main purpose of this paper is to provide a general formula for $\mathcal{K}_4(G_k(q))$, thus extending the results of Evans, Pulham and Sheehan to generalized Paley graphs and to prime powers. 
In the same way that a zero value for $\mathcal{K}_4(G(q))$ means $q$ is a strict lower bound for the two-color diagonal Ramsey number $R(4,4)$, we show that zero values for $\mathcal{K}_4(G_k(q))$ yield lower bounds for the multicolor diagonal Ramsey numbers $R_k(4)=R(4,4,\cdots,4)$. 
We also provide a general formula $\mathcal{K}_3(G_k(q))$ and give lower bounds for the multicolor diagonal Ramsey numbers $R_k(3)=R(3,3,\cdots,3)$. In both cases, we state explicitly these lower bounds for small $k$ and compare to other known bounds.

\section{Statement of Results}\label{sec_Results}
We present our results in two parts, the first relating to complete subgraphs of order four and then those relating to complete subgraphs of order three.
\subsection{Complete Subgraphs of Order Four}\label{sec_Results_Four}
Many of our results in this section will be stated in terms of Greene's finite field hypergeometric function \cite{G, G2}. Let $\widehat{\mathbb{F}^{*}_{q}}$ denote the group of multiplicative characters of $\mathbb{F}^{*}_{q}$. We extend the domain of $\chi \in \widehat{\mathbb{F}^{*}_{q}}$ to $\mathbb{F}_{q}$, by defining $\chi(0):=0$ (including the trivial character $\varepsilon$) and denote $\bar{\chi}$ as the inverse of $\chi$. We let $\varphi \in \widehat{\mathbb{F}^{*}_{q}}$ be the character of order two. For characters $A$ and $B$ of $\mathbb{F}_{q}^*$ we define the usual Jacobi sum $J(A, B):=\sum_{a \in \mathbb{F}_{q}} A(a) B(1-a)$ and define the symbol
$\binom{A}{B} := \frac{B(-1)}{q} J(A, \bar{B})$.
For characters $A_0,A_1,\dotsc, A_n$ and $B_1, \dotsc, B_n$ of $\mathbb{F}_{q}^*$ and 
$\lambda \in \mathbb{F}_{q}$, define the finite field hypergeometric function
\begin{equation*}
{_{n+1}F_n} {\left( \begin{array}{cccc} A_0, & A_1, & \dotsc, & A_n \\
\phantom{A_0} & B_1, & \dotsc, & B_n \end{array}
\Big| \; \lambda \right)}_{q}
:= \frac{q}{q-1} \sum_{\chi} \binom{A_0 \chi}{\chi} \binom{A_1 \chi}{B_1 \chi}
\dotsm \binom{A_n \chi}{B_n \chi} \chi(\lambda),
\end{equation*}
where the sum is over all multiplicative characters $\chi$ of $\mathbb{F}_{q}^*$.
Let $k \geq 2$ be an integer, let $q\equiv 1 \imod {k}$ be a prime power and let $\chi_k \in \widehat{\mathbb{F}^{*}_{q}}$ be a character of order $k$.
For $\vec{t}=(t_1,t_2,t_3,t_4,t_5) \in \left({\mathbb{Z}_{k}}\right)^{5}$ we define
\begin{equation*}
{_{3}F_2} \left( \vec{t} \; \big| \;  \lambda \right)_{q,k}
:=
{_{3}F_2}\biggl( \begin{array}{ccc} \chi_k^{t_1}, & \chi_k^{t_2}, & \chi_k^{t_3} \vspace{.05in}\\
\phantom{\chi_k^{t_1}} & \chi_k^{t_4}, & \chi_k^{t_5} \end{array}
\Big| \; \lambda \biggr)_{q}.
\end{equation*}

In our first result, we show that $\mathcal{K}_4(G_k(q))$ can be written quite simply in terms of ${_{3}F_2}$ finite field hypergeometric functions.
\begin{theorem}\label{thm_Main1}
Let $k \geq 2$ be an integer. Let $q$ be a prime power such that $q \equiv 1 \imod {k}$ if $q$ is even, or, $q \equiv 1 \imod {2k}$ if $q$ is odd. Then
\begin{equation*}
\mathcal{K}_4(G_k(q)) = 
\frac{q^3(q-1)}{24 \cdot k^6} 
\sum_{\vec{t} \in \left({\mathbb{Z}_{k}}\right)^{5}} {_{3}F_2} \left( \vec{t} \; \big| \;  1 \right)_{q,k}.
\end{equation*}
\end{theorem}
\noindent 
Many of the summands in Theorem \ref{thm_Main1} can be simplified using known reduction formulae for finite field hypergeometric functions. Splitting off these terms yields our second result. 

\begin{theorem}\label{thm_Main2}
Let $k \geq 2$ be an integer. Let $q$ be a prime power such that $q \equiv 1 \imod {k}$ if $q$ is even, or, $q \equiv 1 \imod {2k}$ if $q$ is odd.  Then
\begin{multline*}
\mathcal{K}_4(G_k(q)) =
\frac{q(q-1)}{24 \cdot k^6} 
\Biggl[ 
10 \, \mathbb{R}_k(q)^2 + 5 \left( q-2k^2+1\right) \mathbb{R}_k(q)  -15 \, \mathbb{S}_k(q) +q^2
\\
-5 \left(2k^2-3k+2 \right) q
+ 15k^3-10k^2+1
+q^2 \sum_{\vec{t} \in X_k} {_{3}F_2} \left( \vec{t} \; \big| \;  1 \right)_{q,k}
\Biggr],
\end{multline*}
where $X_k := \{  (t_1,t_2,t_3,t_4,t_5) \in \left({\mathbb{Z}_{k}}\right)^{5} \mid t_1,t_2,t_3 \neq 0, t_4,t_5 \, ; \, t_1+t_2+t_3 \neq t_4+t_5 \}$,
\begin{equation*}
\mathbb{R}_k(q) := \sum_{\substack{s,t=1 \\s+t \not\equiv 0 \, (k)}}^{k-1} 
J \left( \chi_k^s, \chi_k^t \right)
\qquad \textup{and} \qquad
\mathbb{S}_k(q) := 
\sum_{\substack{s,t,v=1 \\ s+t, v+t, v-s \not\equiv 0 \, (k) }}^{k-1} 
J \left( \chi_k^s, \chi_k^t \right)
J \left(\bar{\chi_k}^s , \chi_k^v \right).
\end{equation*}
\end{theorem}
\noindent
This looks quite messy, which is the price we pay for a formula which holds for all $k$, but for a given $k$ it tidies up somewhat. 
Many of the summands that still remain in Theorem \ref{thm_Main2} are equal and we can establish an equivalence relation on the set $X_k$ to reduce the number of hypergeometric terms to equivalence class representatives. We discuss this process in Section \ref{sec_Orbits}. 
In particular, for small $k$, the formula reduces to relatively few terms, as we see in the following results.

\begin{cor}[${k=2}$]\label{cor_k2}
Let $q=p^r\equiv 1 \pmod {4}$ for a prime $p$. Write $q=x^2 +y^2$ for integers $x$ and $y$, such that $y$ is even, and $p \nmid x$ when $p \equiv 1 \imod{4}$. Then
\begin{equation*}
\mathcal{K}_4(G(q)) =
\frac{q(q-1)((q-9)^2-4y^2)}{2^9 \cdot 3}. 
\end{equation*}
Note that $y=0$ if (and only if) $p \equiv 3 \pmod 4$.
\end{cor}
\noindent
An inductive algorithm for finding $x$ and $y$, when $p \equiv 1 \imod{4}$, is described in \cite[Prop 3.2]{McC11}. Corollary \ref{cor_k2} extends the result of Evans, Pulham and Sheehan, (\ref{for_EPS}) above, to prime powers. It is easy to see that $\mathcal{K}_4(G(17)) =0$ and so we reconfirm the lower bound for $R(4,4)$ of Greenwood and Gleason.

\begin{cor}\label{cor_Ramseyk2}
$18 \leq R(4,4)$.
\end{cor}

\begin{cor}[${k=3}$]\label{cor_k3}
Let $q=p^r$ for a prime $p$, such that $q \equiv 1 \pmod {3}$ if $q$ is even, or, $q \equiv 1 \pmod {6}$ if $q$ is odd.  Let $\chi_3 \in \widehat{\mathbb{F}^{*}_{q}}$ be a character of order $3$. When $p \equiv 1 \imod{3}$, write $4q=c^2 +3d^2$ for integers $c$ and $d$, such that $c \equiv 1 \imod 3$, $d \equiv 0 \imod 3$ and $p \nmid c$. When $p \equiv 2 \imod{3}$, let $c=-2(-p)^{\frac{r}{2}}$. Then
\begin{equation*}
\mathcal{K}_4(G_3(q)) =
\frac{q(q-1)}{2^3 \cdot 3^7}. 
\Biggl[ 
q^2+5q(c-11)+10c^2-85c+316
+12 \, q^2 
{_{3}F_2} {\left( \begin{array}{cccc} \chi_3, & \chi_3, &  \bar{\chi_3} \\
\phantom{\chi_3} & \varepsilon, &  \varepsilon \end{array}
\Big| \; 1 \right)}_{q}
\Biggr].
\end{equation*}
\end{cor}

\begin{cor}\label{cor_Ramseyk3}
$128 \leq R(4,4,4)$.
\end{cor}
\noindent
It is known that $128 \leq R(4,4,4) \leq 230$ \cite {HI, R}. So again the (generalized) Paley graph matches the best known lower bound. However, this is no longer the case when $k=4$.

\begin{cor}[${k=4}$]\label{cor_k4}
Let $q=p^r\equiv 1 \pmod {8}$ for a prime $p$. Let $\varphi, \chi_4  \in \widehat{\mathbb{F}^{*}_{q}}$ be characters of order $2$ and $4$ respectively.
Write $q=x^2 +y^2$ for integers $x$ and $y$, such that $x \equiv 1 \pmod{4}$, and $p \nmid x$ when $p \equiv 1 \imod{4}$.
Write $q=u^2+2v^2$ for integers $u$ and $v$, such that $u \equiv 3 \pmod{4}$, and $p \nmid u$ when $p \equiv 1,3 \imod{8}$.
Then
\begin{multline*}
\mathcal{K}_4(G_4(q)) =
\frac{q(q-1)}{2^{15} \cdot 3}. 
\Biggl[ 
q^2-2q(15x+101)+304x^2+(930-40u)x+801+120u^2 \\
+12 \, q^2 
{_{3}F_2} {\left( \begin{array}{cccc} \chi_4, & \chi_4, &  \bar{\chi_4} \\
\phantom{\chi_3} & \varepsilon, &  \varepsilon \end{array}
\Big| \; 1 \right)}_{q}
+30 \, q^2 
{_{3}F_2} {\left( \begin{array}{cccc}  \chi_4, & \varphi, &  \varphi \\
\phantom{\chi_4} & \varepsilon, &  \varepsilon \end{array}
\Big| \; 1 \right)}_{q}
\Biggr].
\end{multline*}
\end{cor}
\noindent 

\begin{cor}\label{cor_Ramseyk4}
$458 \leq R(4,4,4,4)$.
\end{cor}
\noindent
This falls short of the best known bound of $634 \leq R(4,4,4,4)$ \cite{XZER}. We discuss the relationship between the generalized Paley graphs and multicolor Ramsey numbers in Section \ref{sec_Ramsey}. When $q=p$ is prime, the values of the three hypergeometric functions in Corollaries \ref{cor_k3} \& \ref{cor_k4} correspond to $p$-th Fourier coefficients of certain non-CM modular forms. We discuss these relationships in Section \ref{sec_ModularForms}.


\subsection{Complete Subgraphs of Order Three}\label{sec_Results_Three}
We first provide a simple formula for $\mathcal{K}_3(G_k(q))$ in terms of Jacobi sums.
\begin{theorem}\label{thm_Main3}
Let $k \geq 2$ be an integer. Let $q$ be a prime power such that $q \equiv 1 \imod {k}$ if $q$ is even, or, $q \equiv 1 \imod {2k}$ if $q$ is odd. Then
\begin{equation*}
\mathcal{K}_3(G_k(q)) = 
\frac{q(q-1)}{6k^3} \, ( \mathbb{R}_k(q) + q -3k+1),
\end{equation*}
where $\mathbb{R}_k(q)$ is as defined in Theorem \ref{thm_Main2}.
\end{theorem}
\noindent
For small $k$, Theorem \ref{thm_Main3} simplifies to simple closed formulae. 
\begin{cor}[${k=2}$]\label{cor_3k2}
Let $q=p^r\equiv 1 \pmod {4}$ for a prime $p$. Then
\begin{equation*}
\mathcal{K}_3(G(q)) =
\frac{q(q-1)(q-5)}{2^4 \cdot 3}. 
\end{equation*}
\end{cor}

\begin{cor}[${k=3}$]\label{cor_3k3}
Let $q=p^r$ for a prime $p$, such that $q \equiv 1 \pmod {3}$ if $q$ is even, or, $q \equiv 1 \pmod {6}$ if $q$ is odd. When $p \equiv 1 \imod{3}$, write $4q=c^2 +3d^2$ for integers $c$ and $d$, such that $c \equiv 1 \imod 3$, $d \equiv 0 \imod 3$ and $p \nmid c$. When $p \equiv 2 \imod{3}$, let $c=-2(-p)^{\frac{r}{2}}$. Then
\begin{equation*}
\mathcal{K}_3(G_3(q)) =
\frac{q(q-1)(q+c-8)}{2 \cdot 3^4}.
\end{equation*}
\end{cor}

\begin{cor}[${k=4}$]\label{cor_3k4}
Let $q=p^r\equiv 1 \pmod {8}$ for a prime $p$.
Write $q=x^2 +y^2$ for integers $x$ and $y$, such that $x \equiv 1 \pmod{4}$, and $p \nmid x$ when $p \equiv 1 \imod{4}$.
Then
\begin{equation*}
\mathcal{K}_3(G_4(q)) =
\frac{q(q-1)(q-6x-11)}{2^{7} \cdot 3}. 
\end{equation*}
\end{cor}
\noindent 
It is easy to see from Corollaries \ref{cor_3k2}-\ref{cor_3k4} that $\mathcal{K}_3(G(5)) = \mathcal{K}_3(G_3(16)) = \mathcal{K}_3(G_4(41)) = 0$ which leads to the following corresponding lower bounds for multicolor Ramsey numbers.
\begin{cor}\label{cor_Ramsey3k234}
$6 \leq R(3,3)$,
$17 \leq R(3,3,3)$, and
$42 \leq R(3,3,3,3)$.
\end{cor}
\noindent
It is known that $R(3,3)=6$ and $R(3,3,3) =17$ \cite {GG}. However, the bound for $R(3,3,3,3)$ implied by the Paley graph falls short of the best known bound of $51 \leq R(3,3,3,3) \leq 62$ \cite{C,FKR}.
As mentioned in Section \ref{sec_Results_Four}, we will discuss the relationship between the generalized Paley graphs and multicolor Ramsey numbers in Section \ref{sec_Ramsey}. Also, when $q$ is prime, $c$ and $x$ in Corollaries \ref{cor_3k3} and \ref{cor_3k4} can be related to the Fourier coefficients of certain modular forms.   
 We discuss these relationships in Section \ref{sec_ModularForms}.
 
%



\section{Preliminaries}\label{sec_Prelim}

\subsection{Jacobi Sums}\label{sec_Prelim_GJSums}
We first recall some well-known properties of Jacobi sums. For further details see \cite{BEW}, noting that we have adjusted results therein to take into account $\varepsilon(0)=0$. 

\begin{prop}\label{prop_JacBasic}
For non-trivial $\chi \in \widehat{\mathbb{F}^{*}_{q}}$ we have
\begin{itemize}
\item[(a)] $J(\varepsilon,\varepsilon)=q-2$;
\item[(b)] $J(\varepsilon, \chi)=-1$; and
\item[(c)] $J(\chi,\bar{\chi}) = -\chi(-1)$.
\end{itemize}
\end{prop}

\begin{prop}\label{prop_JacXfer}
For $\chi, \psi \in \widehat{\mathbb{F}^{*}_{q}}$,
$J(\chi, \psi) = \chi(-1) J(\chi, \bar{\chi} \bar{\psi})$.
\end{prop}

\begin{prop}\label{prop_JacProdq}
For non-trivial $\chi, \psi \in \widehat{\mathbb{F}^{*}_{q}}$ with $\chi \psi$ non-trivial,
$J(\chi, \psi) J(\bar{\chi}, \bar{\psi}) =q$.
\end{prop}

\noindent
Recall, if we let $k \geq 2$ be an integer, $q\equiv 1 \pmod {k}$ be a prime power and $\chi_k \in \widehat{\mathbb{F}^{*}_{q}}$ be a character of order $k$,
then for $b \in \mathbb{F}_q^*$, we have the orthogonal relation \cite[p11]{BEW}
\begin{equation}\label{for_OrthRel}
\frac{1}{k} \sum_{t=0}^{k-1} \chi_k^t(b)
=
\begin{cases}
1 & \textup{ if $b$ is a $k$-th power},\\
0 & \textup{ if $b$ is not a $k$-th power}.
\end{cases}
\end{equation}

We now develop some preliminary results which will be used in later sections. As a straightforward consequence of Propositions \ref{prop_JacBasic} and \ref{prop_JacProdq} we see that
\begin{equation}\label{for_JJ1conj}
\sum_{s,t=1}^{k-1} J(\chi_k^s,\chi_k^t) J(\bar{\chi_k}^s,\bar{\chi_k}^t) = (k-1) \left[(k-2)q+1 \right].
\end{equation}
We define the quantities
\begin{equation*}
\mathbb{J}_0(q,k) := \sum_{s,t=0}^{k-1} J(\chi_k^s,\chi_k^t) \qquad \textup{and} \qquad \mathbb{JJ}_0(q,k) := \sum_{s,t,v=0}^{k-1} J(\chi_k^s,\chi_k^t) J(\bar{\chi_k}^s,\chi_k^v)
\end{equation*}
which appear often in our reckonings. Using Propositions \ref{prop_JacBasic} \& \ref{prop_JacXfer} and (\ref{for_JJ1conj}) it is a straightforward exercise to show that
\begin{equation}\label{for_J0R}
\mathbb{J}_0(q,k) = \mathbb{R}_k(q) +q-3k+1
\end{equation}
and
\begin{equation}\label{for_JJ0S}
\mathbb{JJ}_0(q,k) = \mathbb{S}_k(q) - 4 \, \mathbb{R}_k(q) + q^2 + k(k-5)q +k^2+6k-3,
\end{equation}
where $\mathbb{R}_k(q)$ and $\mathbb{S}_k(q)$ are as defined in Theorem \ref{thm_Main2}. 

\begin{lemma}\label{lem_JacOrder3}
Let $q=p^r$ for a prime $p$, such that $q \equiv 1 \pmod {3}$ if $q$ is even, or, $q \equiv 1 \pmod {6}$ if $q$ is odd.  Let $\chi_3 \in \widehat{\mathbb{F}^{*}_{q}}$ be a character of order $3$. When $p \equiv 1 \imod{3}$, write $4q=c^2 +3d^2$ for integers $c$ and $d$, such that $c \equiv 1 \imod 3$, $d\equiv 0 \imod 3$ and $p \nmid c$. When $p \equiv 2 \imod{3}$, let $c=-2(-p)^{\frac{r}{2}}$. Then
$$J(\chi_3,\chi_3) + J(\bar{\chi_3},\bar{\chi_3})=c.$$
\end{lemma}

\begin{proof}
When $p \equiv 1 \imod{3}$, \cite[Prop 1]{PAR} tells us that 
$$J(\chi_3,\chi_3) + J(\bar{\chi_3},\bar{\chi_3}) = \left(\frac{c}{2} + \frac{d}{2}\sqrt{-3}\right) + \left(\frac{c}{2} - \frac{d}{2}\sqrt{-3}\right) = c.$$
Now consider the case when $p \equiv 2 \imod{3}$. Note $r$ is even in this case. Let $\omega =\frac{-1+\sqrt{-3}}{2}$. From \cite[Ch 2]{BEW} we have
\begin{enumerate}
\item $J(\chi_3,\chi_3) \in \mathbb{Z}[\omega]$;
\item $J(\chi_3,\chi_3)J(\bar{\chi_3},\bar{\chi_3}) = q$; and 
\item $J(\chi_3,\chi_3) \equiv -1 \pmod{(1-\omega)^2}$. 
\end{enumerate}
Noting (1) and the fact that $p$ is inert in $\mathbb{Z}[\omega]$, as $p \equiv 2 \imod{3}$, by (2) we must have $(J(\chi_3,\chi_3))=(J(\bar{\chi_3},\bar{\chi_3}) )=(p^{\frac{r}{2}})$ as ideals. Now $(p^{\frac{r}{2}})=(-(-p)^{\frac{r}{2}})$ and 
\begin{align*}
-(-p)^{\frac{r}{2}} 
&\equiv -1 \pmod{3}\\
&\equiv -1 \pmod{(1-\omega)^2}
\end{align*}
as $3= (1-\omega)^2 (1+\omega)$. Then by \cite[Lem 5]{PAR}, $J(\chi_3,\chi_3) = J(\bar{\chi_3}, \bar{\chi_3}) = -(-p)^{\frac{r}{2}}$ and so $J(\chi_3,\chi_3) + J(\bar{\chi_3},\bar{\chi_3})  = -2(-p)^{\frac{r}{2}} = c$. 
\end{proof}

\begin{lemma}[\cite{KR}]\label{lem_JacOrder4}
Let $q=p^r\equiv 1 \pmod {4}$ for a prime $p$. Write $q=x^2 +y^2$ for integers $x$ and $y$, such that $x\equiv 1 \pmod{4}$, and $p \nmid x$ when $p \equiv 1 \imod{4}$. Then
\begin{enumerate}
\item $J(\chi_4,\chi_4) + J(\bar{\chi_4},\bar{\chi_4}) = -2x$; and
\item $J(\chi_4,\chi_4)^2 + J(\bar{\chi_4},\bar{\chi_4})^2 = 2x^2-2y^2 = 4x^2-2q = 2q-4y^2$.
\end{enumerate}
\end{lemma}
 
\begin{proof}
Let $q=x^2 +y^2$ for integers $x$ and $y$, such that $x \equiv 1 \imod{4}$ and $y$ is even, and $p \nmid x$ when $p \equiv 1 \imod{4}$. Then \cite[Props 1 \& 2]{KR} tells us that, when $p \equiv 1 \imod{4}$, $x$ is uniquely determined, and $y$ up to sign; when $p \equiv 3 \imod{4}$ that $x=(-p)^{\frac{r}{2}}, y=0$ is the only solution; and
$J(\chi_4,\chi_4) = -x \pm iy$. The results follow as $J(\bar{\chi_4},\bar{\chi_4})$ is the complex conjugate of $J(\chi_4,\chi_4)$. 
 \end{proof}

\begin{lemma}\label{lem_JacOrder8}
Let $q=p^r\equiv 1 \pmod {8}$ for a prime $p$. 
Write $q=x^2 +y^2$ for integers $x$ and $y$, such that $x \equiv 1 \pmod{4}$, and $p \nmid x$ when $p \equiv 1 \imod{4}$.
Write $q=u^2+2v^2$ for integers $u$ and $v$, such that $u \equiv 3 \pmod{4}$, and $p \nmid u$ when $p \equiv 1,3 \imod{8}$.
Then
\begin{enumerate}
\item $J(\chi_8,\chi_8) = J(\chi_8^3,\chi_8^3) = \chi_8(-4) J(\chi_8,\chi_8^3)$;
\item $Re(J(\chi_8,\chi_8))= \chi_8(4) \, u$;
\item $J(\chi_8,\chi_8^2) = \chi_8(-4) J(\chi_4,\chi_4)$; and
\item $Re(J(\chi_8,\chi_8^2))= - \chi_8(-4) \, x$.
\end{enumerate}
\end{lemma}

\begin{proof}
(1) This follows from \cite[Thm 2.1.6]{BEW}.
(2) The opening arguments of the proof of \cite[Thm 3.3.1]{BEW} also apply over $\mathbb{F}_q$, so $J(\chi_8,\chi_8) = \chi_8(4) (u + v \sqrt{-2})$ for integers $u$ and $v$, such that $q=u^2+2v^2$ and $u \equiv 3 \imod{4}$, and $\chi_8(4)=\pm1$. Similar arguments to those in the proof of \cite[Props 1 \& 2]{KR} can then be applied to tell us that when $p$ is inert in $\mathbb{Z}[\sqrt{-2}]$, i.e. when $p \equiv 5,7 \imod{8}$, $u=\pm p^{\frac{r}{2}}, v=0$ is the only solution; and when $p$ splits in $\mathbb{Z}[\sqrt{-2}]$, i.e. when $p \equiv 1,3 \imod{8}$, $p \nmid u$ and this uniquely determines $u$, and $v$ up to sign.
(3) This is proved in \cite[Thm 3.3.3]{BEW}.
(4) This follows from (3) and Lemma \ref{lem_JacOrder4}.
\end{proof}


\subsection{Properties of Finite Field Hypergeometric Functions}\label{sec_Prelim_HypFns}
As we have seen in Section \ref{sec_Results}, our most general results are expressed in terms of the finite field hypergeometric functions of Greene \cite{G, G2}.
These functions have very nice expressions as character sums \cite[Def 3.5 (after change of variable), Cor 3.14]{G2}.
For characters $A, B, C, D, E$ of $\mathbb{F}_{q}^*$,
\begin{equation}\label{for_CharSum2F1}
q \, {_{2}F_1} {\left( \begin{array}{cc} A, & B \\
\phantom{A} & C \end{array}
\Big| \; \lambda \right)}_{q}
= \sum_{b \in \mathbb{F}_q} A\bar{C}(b) \bar{B}C(1-b)  \bar{A} (b- \lambda)
\end{equation}
and
\begin{equation}\label{for_CharSum3F2}
q^2 \, {_{3}F_2} {\left( \begin{array}{ccc} A, & B, &  C \\
\phantom{A} & D, & E \end{array}
\Big| \; \lambda \right)}_{q}
=  \sum_{a,b \in \mathbb{F}_q} A\bar{E}(a) \bar{C}E(1-a) B(b) \bar{B}D(b-1) \bar{A} (a- \lambda b)
\end{equation}
Much of our work in this paper relies on being able to simplify expressions involving these functions. We will use the following reduction formulae \cite[Thms 3.15 \& 4.35]{G2} in our proof of Theorem \ref{thm_Main2}. For characters $A, B, C, D, E$ of $\mathbb{F}_{q}^*$,
\begin{align}
\label{for_3F2Red_1}
{_{3}F_2} {\left( \begin{array}{ccc} \varepsilon, & B, &  C \\
\phantom{A} & D, & E \end{array}
\Big| \; 1 \right)}_{q}
=&
-\frac{1}{q} \,
{_{2}F_1} {\left( \begin{array}{cc} B\bar{D}, &  C\bar{D} \\
\phantom{B\bar{D}} & E\bar{D} \end{array}
\Big| \; 1 \right)}_{q}
+ \binom{B}{D} \binom{C}{E};
\\[6pt] \label{for_3F2Red_2}
{_{3}F_2} {\left( \begin{array}{ccc} A, & \varepsilon, &  C \\
\phantom{A} & D & E \end{array}
\Big| \; 1 \right)}_{q}
=&
A(-1) \binom{D}{A}
{_{2}F_1} {\left( \begin{array}{cc} A\bar{D}, &  C\bar{D} \\
\phantom{A\bar{D}} & E\bar{D} \end{array}
\Big| \; 1 \right)}_{q}
-\frac{D(-1)}{q} \binom{C}{E};
\\[6pt] \label{for_3F2Red_3}
{_{3}F_2} {\left( \begin{array}{ccc} A, & B, &  C \\
\phantom{A} & A, & E \end{array}
\Big| \; 1 \right)}_{q}
=&
\binom{B}{A}
{_{2}F_1} {\left( \begin{array}{cc} B, &  C\\
\phantom{B} & E \end{array}
\Big| \; 1 \right)}_{q}
-\frac{\bar{A}(-1)}{q} \binom{C\bar{A}}{E\bar{A}};
\\[6pt] \label{for_3F2Red_4}
{_{3}F_2} {\left( \begin{array}{ccc} A, & B, &  C \\
\phantom{A} & B, & E \end{array}
\Big| \; 1 \right)}_{q}
=&
-\frac{1}{q} \,
{_{2}F_1} {\left( \begin{array}{cc} A, &  C \\
\phantom{A} & E \end{array}
\Big| \; 1 \right)}_{q}
+ \binom{A\bar{B}}{\bar{B}} \binom{C\bar{B}}{E\bar{B}};
\\[6pt] \label{for_3F2Red_5}
{_{3}F_2} {\left( \begin{array}{ccc} A, & B, &  C \\
\phantom{A} & D, & B \end{array}
\Big| \; 1 \right)}_{q}
=&
\binom{C\bar{D}}{B\bar{D}}
{_{2}F_1} {\left( \begin{array}{cc} A, &  C\\
\phantom{A} & D \end{array}
\Big| \; 1 \right)}_{q}
-\frac{BD(-1)}{q} \binom{A\bar{B}}{\bar{B}}; \, \textup{and}
\\[6pt] \label{for_3F2Red_6}
{_{3}F_2} {\left( \begin{array}{ccc} A, & B, &  C \\
\phantom{A} & D, & ABC\bar{D} \end{array}
\Big| \; 1 \right)}_{q}
&=
BC(-1)\binom{C}{D\bar{A}}\binom{B}{D\bar{C}}
-\frac{BD(-1)}{q} \binom{D\bar{B}}{A}.
\end{align}
It is easy to see from their definition that the value of finite field hypergeometric functions is invariant under permuting columns of parameters, i.e,
\begin{equation}\label{for_3F2Per}
{_{3}F_2} {\left( \begin{array}{ccc} A, & B, &  C \\
\phantom{A} & D, & E \end{array}
\Big| \; 1 \right)}_{q}
=
{_{3}F_2} {\left( \begin{array}{ccc} A, & C, &  B \\
\phantom{A} & E, & D \end{array}
\Big| \; 1 \right)}_{q},
\end{equation}
so the reduction formulae (\ref{for_3F2Red_2})-(\ref{for_3F2Red_5}) can each be applied to two different parameter structures. 
The ${_{2}F_1}( \cdot |1)$'s appearing in (\ref{for_3F2Red_1})-(\ref{for_3F2Red_5}) can also be reduced using \cite[Theorem 4.9]{G2}
\begin{equation}\label{for_2F1Red}
{_{2}F_1} {\left( \begin{array}{cc} A, &  B \\
\phantom{A} & C \end{array}
\Big| \; 1 \right)}_{q}
= A(-1) \binom{B}{\bar{A}C}.
\end{equation}

Another important feature of finite field hypergeometric functions is transformation formulae relating the values of functions with different parameters (analogous to those for classical hypergeometric series).  Of interest to us are the following ${_{3}F_2}( \cdot |1)$ transformations which can be 
found in \cite{G2}
but are stated more succinctly in \cite[Thms 5.14, 5.18 \& 5.20]{G}.
\begin{align}
\label{for_3F2_T1}
{_{3}F_2} {\left( \begin{array}{ccc} A, & B, &  C \\
\phantom{A} & D, & E \end{array}
\Big| \; 1 \right)}_{q}
&=
{_{3}F_2} {\left( \begin{array}{ccc} B\bar{D}, & A\bar{D}, &  C\bar{D} \\
\phantom{A} & \bar{D}, & E\bar{D} \end{array}
\Big| \; 1 \right)}_{q};
\\[6pt] \label{for_3F2_T2}
{_{3}F_2} {\left( \begin{array}{ccc} A, & B, &  C \\
\phantom{A} & D, & E \end{array}
\Big| \; 1 \right)}_{q}
&=
ABCDE(-1) \,
{_{3}F_2} {\left( \begin{array}{ccc} A, & A\bar{D}, &  A\bar{E} \\
\phantom{A} & A\bar{B}, & A\bar{C} \end{array}
\Big| \; 1 \right)}_{q};
\\[6pt] \label{for_3F2_T3}
{_{3}F_2} {\left( \begin{array}{ccc} A, & B, &  C \\
\phantom{A} & D, & E \end{array}
\Big| \; 1 \right)}_{q}
&=
ABCDE(-1) \,
{_{3}F_2} {\left( \begin{array}{ccc} B\bar{D}, & B, &  B\bar{E} \\
\phantom{A} & B\bar{A}, & B\bar{C} \end{array}
\Big| \; 1 \right)}_{q};
\\[6pt] \label{for_3F2_T4}
{_{3}F_2} {\left( \begin{array}{ccc} A, & B, &  C \\
\phantom{A} & D, & E \end{array}
\Big| \; 1 \right)}_{q}
&=
AE(-1) \,
{_{3}F_2} {\left( \begin{array}{ccc} A, & B, &  \bar{C}E \\
\phantom{A} & AB\bar{D}, & E \end{array}
\Big| \; 1 \right)}_{q};
\\[6pt] \label{for_3F2_T5}
{_{3}F_2} {\left( \begin{array}{ccc} A, & B, &  C \\
\phantom{A} & D, & E \end{array}
\Big| \; 1 \right)}_{q}
&=
AD(-1) \,
{_{3}F_2} {\left( \begin{array}{ccc} A, & D\bar{B}, &  C \\
\phantom{A} & D, & AC\bar{E} \end{array}
\Big| \; 1 \right)}_{q};
\\[6pt] \label{for_3F2_T6}
{_{3}F_2} {\left( \begin{array}{ccc} A, & B, &  C \\
\phantom{A} & D, & E \end{array}
\Big| \; 1 \right)}_{q}
&=
B(-1) \,
{_{3}F_2} {\left( \begin{array}{ccc} \bar{A}D, & B, &  C \\
\phantom{A} & D, & BC\bar{E} \end{array}
\Big| \; 1 \right)}_{q}; \, \textup{and}
\\[6pt] \label{for_3F2_T7}
{_{3}F_2} {\left( \begin{array}{ccc} A, & B, &  C \\
\phantom{A} & D, & E \end{array}
\Big| \; 1 \right)}_{q}
&=
AB(-1) \,
{_{3}F_2} {\left( \begin{array}{ccc} \bar{A}D, & \bar{B}D, &  C \\
\phantom{A} & D, & \bar{AB}DE \end{array}
\Big| \; 1 \right)}_{q}.
\end{align}


\section{A pair of subgraphs of $G_k(q)$ and Proofs of Theorems \ref{thm_Main1} and \ref{thm_Main3}}\label{sec_Subgraphs}
In this section we define two subgraphs of the generalized Paley graph $G_k(q)$ and relate the number of complete subgraphs of a given order for each graph, which we use to prove Theorems \ref{thm_Main1} and \ref{thm_Main3}. 
For a graph $G$ we denote its vertex set by $V(G)$ and its edge set by $E(G)$, so the order of $G$ is $\#V(G)$ and the size of $G$ is $\#E(G)$. For a given vertex $a$ of $G$ we denote the degree of $a$ in $G$ by $\deg_G(a)$. It is easy to see from its definition that $\#V(G_k(q)) =q$, $\deg_{G_k(q)}(a) = \frac{q-1}{k}$ for all vertices $a$, and, consequently, $\#E(G_k(q)) = \frac{q(q-1)}{2k}$.

Let $H_k(q)$ be the induced subgraph of $G_k(q)$ whose vertex set is $S_k$, the set of $k$-th power residues of $\mathbb{F}_q$. Therefore, $\#V(H_k(q))=|S_k| = \frac{q-1}{k}$. 
Now 
$$ab \in E(H_k(q)) \Longleftrightarrow \chi_k(a)=\chi_k(b)=\chi_k(a-b)=1.$$
So, for $a \in V(H_k(q))$, using (\ref{for_OrthRel}), we get that
\begin{align*}
\deg_{H_k(q)}(a) 
&= \frac{1}{k^2} \sum_{b \in \mathbb{F}_q^* \setminus \{a\}} \sum_{s=0}^{k-1} \chi_k^s(b) \sum_{t=0}^{k-1} \chi_k^t(a-b)\\
&= \frac{1}{k^2} \sum_{s,t=0}^{k-1} \chi_k^{s+t}(a) J(\chi_k^s, \chi_k^t)\\
&= \frac{1}{k^2} \, \mathbb{J}_0(q, k)\\
&= \frac{1}{k^2} ( \mathbb{R}_k(q) + q -3k+1) \qquad \textup{(by (\ref{for_J0R}))}.
\end{align*}
This is independent of $a$ so 
\begin{align*}
\#E(H_k(q))
&= \frac{1}{2} \cdot  \#V(H_k(q)) \cdot \deg_{H_k(q)}(a)\\
&= \frac{q-1}{2k^3} \, \mathbb{J}_0(q, k)\\
&= \frac{q-1}{2k^3} \, ( \mathbb{R}_k(q) + q -3k+1) .
\end{align*}

Let $H^{1}_k(q)$ be the induced subgraph of $H_k(q)$ whose vertex set is the set of neighbors of 1 in $H_k(q)$. Therefore 
$$\#V(H^{1}_k(q)) = \deg_{H_k(q)}(1) =  \frac{1}{k^2} \, \mathbb{J}_0(q, k) = \frac{1}{k^2} ( \mathbb{R}_k(q) + q -3k+1),$$
and 
$$ab \in E(H^{1}_k(q)) \Longleftrightarrow \chi_k(a)=\chi_k(b)=\chi_k(1-a)=\chi_k(1-b)=\chi_k(a-b)=1.$$
Again using (\ref{for_OrthRel}), and noting that $\chi_k(-1)=1$, we get that for $a \in V(H^{1}_k(q))$,
\begin{align*}
\deg_{H^{1}_k(q)}(a) 
&= \frac{1}{k^3} \sum_{b \in \mathbb{F}_q^* \setminus \{1,a\}} \sum_{t_1=0}^{k-1} \chi_k^{t_1}(b) \sum_{t_2=0}^{k-1} \chi_k^{t_2}(1-b) \sum_{t_3=0}^{k-1} \chi_k^{t_3}(a-b)\\
&= \frac{1}{k^3} \sum_{t_1,t_2,t_3=0}^{k-1}   \sum_{b \in \mathbb{F}_q^* \setminus \{1,a\}}  \chi_k^{t_1-t_3}(b) \,  \chi_k^{t_3-t_2}(1-b) \, \chi_k^{-t_1}(a-b)\\
&= \frac{1}{k^3} \sum_{t_1,t_2,t_3=0}^{k-1} 
q \, {_{2}F_1} {\left( \begin{array}{ccc} \chi_k^{t_1}, & \chi_k^{t_2} \\[0.05in]
\phantom{\chi_k^{t_1}} & \chi_k^{t_3}  \end{array}
\Big| \; a \right)}_{q}, \qquad \textup{(using (\ref{for_CharSum2F1})),}
\end{align*}
where we have used a change of variables to get the second line. Finally, we get that
\begin{align*}
\#E(H^{1}_k(q))
&= \frac{1}{2} \sum_{a \in V(H^{1}_k(q))} \deg_{H^{1}_k(q)}(a) \\
&=  \frac{1}{2k^3} \sum_{\substack{a \in \mathbb{F}_q \\ \chi_k(a)=\chi_k(1-a)=1}}  \sum_{t_1,t_2,t_3=0}^{k-1}   \sum_{b \in \mathbb{F}_q^* \setminus \{1,a\}}  \chi_k^{t_1}(b) \,  \chi_k^{t_2}(1-b) \, \chi_k^{t_3}(a-b)\\
&=  \frac{1}{2k^5}  \sum_{t_1,t_2,t_3,t_4,t_5=0}^{k-1}  \sum_{\substack{a,b \in \mathbb{F}_q^* \setminus \{1\} \\ a\neq b}}
\chi_k^{t_1}(b) \,  \chi_k^{t_2}(1-b) \, \chi_k^{t_3}(a-b) \,  \chi_k^{t_4}(a) \,  \chi_k^{t_5}(1-a)\\
&=  \frac{1}{2k^5}  \sum_{t_1,t_2,t_3,t_4,t_5=0}^{k-1}  
\sum_{a,b \in \mathbb{F}_q}
\chi_k^{t_1-t_5}(a) \,  \chi_k^{t_5-t_3}(1-a) \, \chi_k^{t_2}(b) \,  \chi_k^{t_4-t_2}(1-b) \,  \chi_k^{-t_1}(a-b)\\
&=  \frac{1}{2k^5}  \sum_{t_1,t_2,t_3,t_4,t_5=0}^{k-1} q^2
{_{3}F_2}\biggl( \begin{array}{ccc} \chi_k^{t_1}, & \chi_k^{t_2}, & \chi_k^{t_3} \vspace{.05in}\\
\phantom{\chi_k^{t_1}} & \chi_k^{t_4}, & \chi_k^{t_5} \end{array}
\Big| \; 1 \biggr)_{q}, \qquad \textup{(using (\ref{for_CharSum3F2})).}
\end{align*}
So we have proved the following proposition.
\begin{prop}\label{prop_Subgraph}
Let $k \geq 2$ be an integer. Let $q$ be a prime power such that $q \equiv 1 \imod {k}$ if $q$ is even, or, $q \equiv 1 \imod {2k}$ if $q$ is odd. Let $H_k(q)$ be the induced subgraph of the generalized Paley graph $G_k(q)$ whose vertex set is the set of $k$-th power residues of $\mathbb{F}_q$. Let $H^{1}_k(q)$ be the induced subgraph of $H_k(q)$ whose vertex set is the set of neighbors of 1 in $H_k(q)$. Then
\begin{enumerate}[\bf (a) ]
\item $\#V(H_k(q)) = \frac{q-1}{k};$\\[-6pt]
\item For $a \in V(H_k(q))$, $\deg_{H_k(q)}(a) = \frac{1}{k^2} \, \mathbb{J}_0(q, k) = \frac{1}{k^2} ( \mathbb{R}_k(q) + q -3k+1)$;\\
\item $\#E(H_k(q)) = \frac{q-1}{2k^3} \, \mathbb{J}_0(q, k) = \frac{q-1}{2k^3} \, ( \mathbb{R}_k(q) + q -3k+1)$;\\
\item $\#V(H^{1}_k(q)) = \frac{1}{k^2} \, \mathbb{J}_0(q, k) = \frac{1}{k^2} ( \mathbb{R}_k(q) + q -3k+1)$;\\
\item For $a \in V(H^{1}_k(q))$, $\deg_{H^{1}_k(q)}(a) = \displaystyle \frac{1}{k^3} \sum_{t_1,t_2,t_3=0}^{k-1} 
q \, {_{2}F_1} {\left( \begin{array}{ccc} \chi_k^{t_1}, & \chi_k^{t_2} \\[0.05in]
\phantom{\chi_k^{t_1}} & \chi_k^{t_3}  \end{array}
\Big| \; a \right)}_{q}$; and\\
\item $\#E(H^{1}_k(q)) = 
\displaystyle  \frac{1}{2k^5} \sum_{t_1,t_2,t_3,t_4,t_5=0}^{k-1} q^2
{_{3}F_2}\biggl( \begin{array}{ccc} \chi_k^{t_1}, & \chi_k^{t_2}, & \chi_k^{t_3} \vspace{.05in}\\
\phantom{\chi_k^{t_1}} & \chi_k^{t_4}, & \chi_k^{t_5} \end{array}
\Big| \; 1 \biggr)_{q}$.
\end{enumerate}
\end{prop}

Next we relate the number of complete subgraphs of a certain order of $G_k(q)$ to those of $H_k(q)$ and $H_k^{1}(q)$.
\begin{lemma}\label{lem_Subgraph}
Let $k, q, G_k(q), H_k(q)$ and $H^{1}_k(q)$ be defined as in Proposition \ref{prop_Subgraph}.
Then, for $n$ a positive integer,
\begin{enumerate}[\bf (a) ]
\item $\mathcal{K}_{n+1}(G_k(q)) = \frac{q}{n+1} \mathcal{K}_{n}(H_k(q))$; and
\item $\mathcal{K}_{n+1}(H_k(q)) = \frac{q-1}{k(n+1)} \mathcal{K}_{n}(H^{1}_k(q))$.
\end{enumerate}
So for $n \geq 2$
\begin{enumerate}[\bf (c) ]
\item $\displaystyle \mathcal{K}_{n+1}(G_k(q)) = \frac{q(q-1)}{kn(n+1)} \mathcal{K}_{n-1}(H^{1}_k(q)).$
\end{enumerate}
\end{lemma}

\begin{proof} Let a complete subgraph of order $m$ be represented by the $m$-tuple of its vertices.
(a) Let $\mathcal{S}_{G,a}$ be the set of complete subgraphs of $G_k(q)$ of order $n+1$ containing the vertex $a$. Let $\mathcal{S}_{H}$ be the set of complete subgraphs of $H_k(q)$ of order $n$. Now
\begin{align*}
(0,a_1,a_2, \cdots,a_n) \in \mathcal{S}_{G,0} 
&\Longleftrightarrow  \chi_k(a_i)=\chi_k(a_i-a_j)=1 \textup{ for all } 1 \leq i < j \leq n\\
&\Longleftrightarrow (a_1,a_2, \cdots,a_n) \in \mathcal{S}_{H} 
\end{align*}
So $|\mathcal{S}_{G,0}| = |\mathcal{S}_{H}|$. For $a \in V(G_k(q))$, the map $f_a(\lambda)=\lambda+a$ is an automorphism of $G_k(q)$, so $|\mathcal{S}_{G,a}| = |\mathcal{S}_{G,0}| = |\mathcal{S}_{H}|$ for all $a \in V(G_k(q))$. Then
$$\mathcal{K}_{n+1}(G_k(q)) = \frac{1}{n+1} \sum_{a \in V(G_k(q))} |\mathcal{S}_{G,a}| = \frac{q}{n+1} |\mathcal{S}_{H}|,$$
as required.\\
(b) Let $\mathcal{S}_{H,a}$ be the set of complete subgraphs of $H_k(q)$ of order $n+1$ containing the vertex $a$. Let $\mathcal{S}_{H^{1}}$ be the set of complete subgraphs of $H^{1}_k(q)$ of order $n$. Now
\begin{align*}
(1,a_1,a_2, \cdots,a_n) \in \mathcal{S}_{H,1} 
&\Longleftrightarrow  \chi_k(a_i)=\chi_k(a_i-1)=\chi_k(a_i-a_j)=1 \textup{ for all } 1 \leq i < j \leq n\\
&\Longleftrightarrow (a_1,a_2, \cdots,a_n) \in \mathcal{S}_{H^{1}} 
\end{align*}
So $|\mathcal{S}_{H,1}| = |\mathcal{S}_{H^{1}}|$. For $a \in V(H_k(q))$, the map $f_a(\lambda)=a \lambda$ is an automorphism of $H_k(q)$, so $|\mathcal{S}_{H,a}| = |\mathcal{S}_{H,1}| = |\mathcal{S}_{H^{1}}|$ for all $a \in V(H_k(q))$. Then
$$\mathcal{K}_{n+1}(H_k(q)) = \frac{1}{n+1} \sum_{a \in V(H_k(q))} |\mathcal{S}_{H,a}| = \frac{\#V(H_k(q))}{n+1} |\mathcal{S}_{H^{1}}|,$$
as required.\\
(c) Follows immediately from combining (a) and (b).
\end{proof}
\noindent
Taking $n=3$ in Lemma \ref{lem_Subgraph} (c) yields Corollary \ref{cor_Subgraph}.
\begin{cor}\label{cor_Subgraph}
Let $k \geq 2$ be an integer. Let $q$ be a prime power such that $q \equiv 1 \imod {k}$ if $q$ is even, or, $q \equiv 1 \imod {2k}$ if $q$ is odd.  Then
$$\mathcal{K}_{4}(G_k(q)) = \frac{q(q-1)}{12k} \#E(H^{1}_k(q)).$$
\end{cor}
\noindent
Combining Corollary \ref{cor_Subgraph} and Proposition \ref{prop_Subgraph} (f) proves Theorem \ref{thm_Main1}.

Taking $n=2$ in Lemma \ref{lem_Subgraph} (a) yields Corollary \ref{cor_Subgraph2}.
\begin{cor}\label{cor_Subgraph2}
Let $k \geq 2$ be an integer. Let $q$ be a prime power such that $q \equiv 1 \imod {k}$ if $q$ is even, or, $q \equiv 1 \imod {2k}$ if $q$ is odd.  Then
$$\mathcal{K}_{3}(G_k(q)) = \frac{q}{3} \#E(H_k(q)).$$
\end{cor}
\noindent
Combining Corollary \ref{cor_Subgraph2} and Proposition \ref{prop_Subgraph} (c) proves Theorem \ref{thm_Main3}.


\section{Proof of Theorem \ref{thm_Main2}}\label{sec_ProofThm2}
We start with Theorem \ref{thm_Main1} and consider  
\begin{equation}\label{for_3F2_Total}
q^2 \sum_{\vec{t} \in \left({\mathbb{Z}_{k}}\right)^{5}} {_{3}F_2} \left( \vec{t} \; \big| \;  1 \right)_{q,k}.
\end{equation}
Many of the summands in (\ref{for_3F2_Total}) can be simplified using the reduction formulae (\ref{for_3F2Red_1})-(\ref{for_3F2Red_6}) described in Section \ref{sec_Prelim_HypFns}. There are ten distinct cases which we will deal with in turn.
\underline{Case 1 ($t_1=0$):} 
We start with the terms which have $t_1 = 0$. These terms can be reduced using (\ref{for_3F2Red_1}) as follows.
\begin{align*}
q^2 \sum_{t_2,t_3,t_4,t_5=0}^{k-1} &
{_{3}F_2}\biggl( \begin{array}{ccc} \varepsilon, & \chi_k^{t_2}, & \chi_k^{t_3} \vspace{.05in}\\
\phantom{\chi_k^{t_1}} & \chi_k^{t_4}, & \chi_k^{t_5} \end{array}
\Big| \; 1 \biggr)_{q}
\\[6pt]
&=
q^2 \sum_{t_2,t_3,t_4,t_5=0}^{k-1}
\left[
-\frac{1}{q} \,
{_{2}F_1} {\left( \begin{array}{cc} \chi_k^{t_2-t_4}, &  \chi_k^{t_3-t_4} \vspace{.05in}\\
\phantom{\chi_k^{t_2-t_4}} &\chi_k^{t_5-t_4}\end{array}
\Big| \; 1 \right)}_{q}
+ \binom{\chi_k^{t_2}}{\chi_k^{t_4}} \binom{\chi_k^{t_3}}{\chi_k^{t_5}} \right]
\\[6pt]
&=
\sum_{t_2,t_3,t_4,t_5=0}^{k-1}
\left[
-q \,
{_{2}F_1} {\left( \begin{array}{cc} \chi_k^{t_2-t_4}, &  \chi_k^{t_3-t_4} \vspace{.05in}\\
\phantom{\chi_k^{t_2-t_4}} &\chi_k^{t_5-t_4}\end{array}
\Big| \; 1 \right)}_{q}
+J(\chi_k^{t_2}, \bar{\chi_k}^{t_4}) \, J(\chi_k^{t_3}, \bar{\chi_k}^{t_5}) \right]
\\[6pt]
&=
-q \sum_{t_2,t_3,t_4,t_5=0}^{k-1}
{_{2}F_1} {\left( \begin{array}{cc} \chi_k^{t_2}, &  \chi_k^{t_3} \vspace{.05in}\\
\phantom{\chi_k^{t_2}} &\chi_k^{t_5}\end{array}
\Big| \; 1 \right)}_{q}
+ \mathbb{J}_0(q, k)^2
\\[6pt]
&=
-qk \sum_{t_2,t_3,t_5=0}^{k-1}
\binom{\chi_k^{t_3}}{\chi_k^{t_5-t_2}} + \mathbb{J}_0(q, k)^2 \qquad \textup{(using (\ref{for_2F1Red}))}
\\[6pt]
&=
-k \sum_{t_2,t_3,t_5=0}^{k-1}
J(\chi_k^{t_3},\chi_k^{t_2-t_5}) + \mathbb{J}_0(q, k)^2 
\\[6pt]
&=
-k \sum_{t_2,t_3,t_5=0}^{k-1}
J(\chi_k^{t_3},\chi_k^{t_2}) + \mathbb{J}_0(q, k)^2 
\\[6pt]
&=
\mathbb{J}_0(q, k)^2-k^2 \, \mathbb{J}_0(q, k), 
\end{align*}
where we have used the fact that $\chi_k(-1)=1$ in many of the steps.\\ 
\underline{Case 2 ($t_2=0$):} Next we reduce the terms which have $t_2 = 0$ using (\ref{for_3F2Red_2}), excluding the $t_1=0$ terms which have already been accounted for.
\begin{align}\notag
q^2 \sum_{\substack{t_1,t_3,t_4,t_5=0 \\ t_1 \neq 0}}^{k-1} &
{_{3}F_2}\biggl( \begin{array}{ccc} \chi_k^{t_1}, & \varepsilon, & \chi_k^{t_3} \vspace{.05in}\\
\phantom{\chi_k^{t_1}} & \chi_k^{t_4}, & \chi_k^{t_5} \end{array}
\Big| \; 1 \biggr)_{q}
\\[6pt] \notag
&=
q^2 \sum_{\substack{t_1,t_3,t_4,t_5=0 \\ t_1 \neq 0}}^{k-1} 
\left[
\binom{\chi_k^{t_4}}{\chi_k^{t_1}}
{_{2}F_1} {\left( \begin{array}{cc} \chi_k^{t_1-t_4}, &  \chi_k^{t_3-t_4} \vspace{.05in}\\
\phantom{\chi_k^{t_1-t_4}} & \chi_k^{t_5-t_4} \end{array}
\Big| \; 1 \right)}_{q}
-\frac{1}{q} \binom{\chi_k^{t_3}}{\chi_k^{t_5}}
\right]
\\[6pt] \notag
&=
\sum_{\substack{t_1,t_3,t_4,t_5=0 \\ t_1 \neq 0}}^{k-1} 
\left[
J(\chi_k^{t_4},\bar{\chi_k}^{t_1}) \, J(\chi_k^{t_3-t_4},\chi_k^{t_1-t_5})
- J(\chi_k^{t_3},\bar{\chi_k}^{t_5})
\right]  \qquad \textup{(using (\ref{for_2F1Red}))}
\\[6pt] \notag
&=
\sum_{\substack{t_1,t_4=0 \\ t_1 \neq 0}}^{k-1} 
J(\chi_k^{t_4},\bar{\chi_k}^{t_1}) \, 
\sum_{t_3,t_5=0}^{k-1} 
J(\chi_k^{t_3-t_4},\chi_k^{t_1-t_5})
- k(k-1) \, \mathbb{J}_0(q, k)
\\[6pt] \notag
&=
\sum_{\substack{t_1,t_4=0 \\ t_1 \neq 0}}^{k-1} 
J(\chi_k^{t_4},\bar{\chi_k}^{t_1}) \, 
\sum_{t_3,t_5=0}^{k-1} 
J(\chi_k^{t_3},\chi_k^{t_5})
- k(k-1) \, \mathbb{J}_0(q, k) 
\\[6pt] \notag
&=
\mathbb{J}_0(q, k)^2
- \mathbb{J}_0(q, k) \sum_{t=0}^{k-1} J(\chi_k^{t},\varepsilon) \, 
- k(k-1) \, \mathbb{J}_0(q, k) 
\\[6pt]
\label{for_RedCase2}
&=
\mathbb{J}_0(q, k)^2
- ((k-1)^2+q-2) \, \mathbb{J}_0(q, k) \qquad \textup{(using Prop \ref{prop_JacBasic})}.
\end{align}
\underline{Case 3 ($t_3=0$):}
Next we reduce the terms which have $t_3 = 0$ using (\ref{for_3F2Red_2}) and (\ref{for_3F2Per}), excluding the $t_1,t_2=0$ terms which have already been accounted for.
\begin{align}
\notag
q^2 & \sum_{\substack{t_1,t_2,t_4,t_5=0 \\ t_1, t_2 \neq 0}}^{k-1} 
{_{3}F_2}\biggl( \begin{array}{ccc} \chi_k^{t_1}, & \chi_k^{t_2}, & \varepsilon \vspace{.05in}\\
\phantom{\chi_k^{t_1}} & \chi_k^{t_4}, & \chi_k^{t_5} \end{array}
\Big| \; 1 \biggr)_{q}
\\[6pt] \notag
&=
q^2 \sum_{\substack{t_1,t_2,t_4,t_5=0 \\ t_1, t_2 \neq 0}}^{k-1} 
{_{3}F_2}\biggl( \begin{array}{ccc} \chi_k^{t_1}, & \varepsilon, & \chi_k^{t_2}  \vspace{.05in}\\
\phantom{\chi_k^{t_1}} & \chi_k^{t_5}, & \chi_k^{t_4} \end{array}
\Big| \; 1 \biggr)_{q}
\\[6pt] \notag
&=
q^2 \sum_{\substack{t_1,t_2,t_4,t_5=0 \\ t_1 \neq 0}}^{k-1} 
{_{3}F_2}\biggl( \begin{array}{ccc} \chi_k^{t_1}, & \varepsilon, & \chi_k^{t_2}  \vspace{.05in}\\
\phantom{\chi_k^{t_1}} & \chi_k^{t_5}, & \chi_k^{t_4} \end{array}
\Big| \; 1 \biggr)_{q}
-
q^2 \sum_{\substack{t_1,t_4,t_5=0 \\ t_1 \neq 0}}^{k-1} 
{_{3}F_2}\biggl( \begin{array}{ccc} \chi_k^{t_1}, & \varepsilon, & \varepsilon \vspace{.05in}\\
\phantom{\chi_k^{t_1}} & \chi_k^{t_5}, & \chi_k^{t_4} \end{array}
\Big| \; 1 \biggr)_{q}
\\[6pt]
&=
\label{for_RedCase3a}
\mathbb{J}_0(q, k)^2
- ((k-1)^2+q-2) \, \mathbb{J}_0(q, k)
-
q^2 \sum_{\substack{t_1,t_4,t_5=0 \\ t_1 \neq 0}}^{k-1} 
{_{3}F_2}\biggl( \begin{array}{ccc} \chi_k^{t_1}, & \varepsilon, & \varepsilon \vspace{.05in}\\
\phantom{\chi_k^{t_1}} & \chi_k^{t_5}, & \chi_k^{t_4} \end{array}
\Big| \; 1 \biggr)_{q}
\end{align}
using the previous case (\ref{for_RedCase2}). Now, using (\ref{for_3F2Red_2}),
\begin{align} \notag
q^2 & \sum_{\substack{t_1,t_4,t_5=0 \\ t_1 \neq 0}}^{k-1} 
{_{3}F_2}\biggl( \begin{array}{ccc} \chi_k^{t_1}, & \varepsilon, & \varepsilon \vspace{.05in}\\
\phantom{\chi_k^{t_1}} & \chi_k^{t_5}, & \chi_k^{t_4} \end{array}
\Big| \; 1 \biggr)_{q}
\\[6pt] \notag
&=
q^2 \sum_{\substack{t_1,t_4,t_5=0 \\ t_1 \neq 0}}^{k-1} 
\left[
\binom{\chi_k^{t_5}}{\chi_k^{t_1}}
{_{2}F_1} {\left( \begin{array}{cc} \chi_k^{t_1-t_5}, &  \bar{\chi_k}^{t_5} \\
\phantom{\chi_k^{t_1-t_5}} & \chi_k^{t_4-t_5} \end{array}
\Big| \; 1 \right)}_{q}
-\frac{1}{q} \binom{\varepsilon}{\chi_k^{t_4}}
\right]
\\[6pt] \notag
&=
\sum_{\substack{t_1,t_4,t_5=0 \\ t_1 \neq 0}}^{k-1} 
\left[
J(\chi_k^{t_5},\bar{\chi_k}^{t_1}) \, J(\bar{\chi_k}^{t_5},\chi_k^{t_1-t_4})
-J(\varepsilon, \bar{\chi_k}^{t_4})
\right]  \qquad \textup{(using (\ref{for_2F1Red}))}
\\[6pt] \notag
&=
\sum_{\substack{t_1,t_4,t_5=0 \\ t_1 \neq 0}}^{k-1} 
J(\chi_k^{t_5},\bar{\chi_k}^{t_1}) \, J(\bar{\chi_k}^{t_5},\chi_k^{t_4})
-k(k-1)(q-k-1) \qquad \textup{(using Prop \ref{prop_JacBasic})}
\\[6pt] \notag
&=
\mathbb{JJ}_0(q,k) - \sum_{t_4,t_5=0}^{k-1} J(\chi_k^{t_5},\varepsilon) \, J(\bar{\chi_k}^{t_5},\chi_k^{t_4})
-k(k-1)(q-k-1)
\\[6pt] \notag
&=
\mathbb{JJ}_0(q,k) 
+ \sum_{\substack{t_4,t_5=0 \\ t_5 \neq 0}}^{k-1}  J(\bar{\chi_k}^{t_5},\chi_k^{t_4})
-(q-2)(q-k-1)-k(k-1)(q-k-1)
\\[6pt]
\label{for_RedCase3b}
&=
\mathbb{JJ}_0(q,k) + \mathbb{J}_0(q,k) 
-(q-k-1)(q+k^2-k-1).
\end{align}
So, combining (\ref{for_RedCase3a}) and (\ref{for_RedCase3b}) we get
\begin{multline*}
q^2  \sum_{\substack{t_1,t_2,t_4,t_5=0 \\ t_1, t_2 \neq 0}}^{k-1} 
{_{3}F_2}\biggl( \begin{array}{ccc} \chi_k^{t_1}, & \chi_k^{t_2}, & \varepsilon \vspace{.05in}\\
\phantom{\chi_k^{t_1}} & \chi_k^{t_4}, & \chi_k^{t_5} \end{array}
\Big| \; 1 \biggr)_{q}\\
=
\mathbb{J}_0(q, k)^2
- ((k-1)^2+q-1) \, \mathbb{J}_0(q, k)
- \mathbb{JJ}_0(q,k) 
+(q-k-1)(q+k^2-k-1).
\end{multline*}
We proceed in this fashion until we have evaluated all terms in (\ref{for_3F2_Total}) that can be reduced using (\ref{for_3F2Red_1})-(\ref{for_3F2Red_6}). Like the first three cases, which we have described above, these evaluations are straightforward and use only basic properties of hypergeometric functions and Jacobi sums from Sections \ref{sec_Prelim_GJSums} and \ref{sec_Prelim_HypFns}. However, the evaluations do become more tedious as we proceed, as we have to exclude successively more cases which have already been accounted for, so we omit the details for reasons of brevity. The remaining cases summarize as follows.
\\[6pt] \underline{Case 4 ($t_1=t_4$):}
\begin{multline*}
q^2  \sum_{\substack{t_1,t_2,t_3,t_5=0 \\ t_1, t_2,t_3 \neq 0}}^{k-1} 
{_{3}F_2}\biggl( \begin{array}{ccc} \chi_k^{t_1}, & \chi_k^{t_2}, & \chi_k^{t_3} \vspace{.05in}\\
\phantom{\chi_k^{t_1}} & \chi_k^{t_1}, & \chi_k^{t_5} \end{array}
\Big| \; 1 \biggr)_{q}\\
=
\mathbb{J}_0(q, k)^2
- (2q+(k-1)^2-4k) \, \mathbb{J}_0(q, k) 
+(q-k-1)(q-3k+1).
\end{multline*}
\\ \underline{Case 5 ($t_1=t_5$):}
\begin{multline*}
q^2  \sum_{\substack{t_1,t_2,t_3,t_4=0 \\ t_1, t_2,t_3 \neq 0 \\ t_1 \neq t_4}}^{k-1} 
{_{3}F_2}\biggl( \begin{array}{ccc} \chi_k^{t_1}, & \chi_k^{t_2}, & \chi_k^{t_3} \vspace{.05in}\\
\phantom{\chi_k^{t_1}} & \chi_k^{t_4}, & \chi_k^{t_1} \end{array}
\Big| \; 1 \biggr)_{q}\\
=
\mathbb{J}_0(q, k)^2
- (2q+k^2-6k+4) \, \mathbb{J}_0(q, k) 
- \mathbb{JJ}_0(q,k) 
+2q^2+(k^2-8k+2)q-k^3+5k^2-2.
\end{multline*}
\\ \underline{Case 6 ($t_2=t_4$):}
\begin{multline*}
q^2  \sum_{\substack{t_1,t_2,t_3,t_5=0 \\ t_1, t_2,t_3 \neq 0 \\ t_1 \neq t_2, t_5}}^{k-1} 
{_{3}F_2}\biggl( \begin{array}{ccc} \chi_k^{t_1}, & \chi_k^{t_2}, & \chi_k^{t_3} \vspace{.05in}\\
\phantom{\chi_k^{t_1}} & \chi_k^{t_2}, & \chi_k^{t_5} \end{array}
\Big| \; 1 \biggr)_{q}\\
=
\mathbb{J}_0(q, k)^2
- (q+k^2-6k+7) \, \mathbb{J}_0(q, k) 
- 2\, \mathbb{JJ}_0(q,k) 
+2q^2+(2k^2-10k+4)q-2k^3+8k^2-4k-2.
\end{multline*}
\\ \underline{Case 7 ($t_3=t_5$):}
\begin{multline*}
q^2  \sum_{\substack{t_1,t_2,t_3,t_4=0 \\ t_1, t_2,t_3 \neq 0 \\ t_1 \neq t_3, t_4 \\ t_2 \neq t_4}}^{k-1} 
{_{3}F_2}\biggl( \begin{array}{ccc} \chi_k^{t_1}, & \chi_k^{t_2}, & \chi_k^{t_3} \vspace{.05in}\\
\phantom{\chi_k^{t_1}} & \chi_k^{t_4}, & \chi_k^{t_3} \end{array}
\Big| \; 1 \biggr)_{q}\\
=
\mathbb{J}_0(q, k)^2
- (2q+k^2-8k+10) \, \mathbb{J}_0(q, k) 
- 2\, \mathbb{JJ}_0(q,k) 
+3q^2+(2k^2-15k+8)q-2k^3+14k^2-15k+1.
\end{multline*}
\\ \underline{Case 8 ($t_2=t_5$):}
\begin{multline*}
q^2  \sum_{\substack{t_1,t_2,t_3,t_4=0 \\ t_1, t_2,t_3 \neq 0 \\ t_1 \neq t_2, t_4 \\ t_2 \neq t_3, t_4}}^{k-1} 
{_{3}F_2}\biggl( \begin{array}{ccc} \chi_k^{t_1}, & \chi_k^{t_2}, & \chi_k^{t_3} \vspace{.05in}\\
\phantom{\chi_k^{t_1}} & \chi_k^{t_4}, & \chi_k^{t_2} \end{array}
\Big| \; 1 \biggr)_{q}\\
=
\mathbb{J}_0(q, k)^2
- (2q+k^2-8k+12) \, \mathbb{J}_0(q, k) 
- 3\, \mathbb{JJ}_0(q,k) 
+4q^2+(3k^2-20k+10)q-3k^3+19k^2-20k+2.
\end{multline*}
\\ \underline{Case 9 ($t_3=t_4$):}
\begin{multline*}
q^2  \sum_{\substack{t_1,t_2,t_3,t_5=0 \\ t_1, t_2,t_3 \neq 0 \\ t_1, t_2 \neq t_3, t_5 \\ t_3 \neq t_5}}^{k-1} 
{_{3}F_2}\biggl( \begin{array}{ccc} \chi_k^{t_1}, & \chi_k^{t_2}, & \chi_k^{t_3} \vspace{.05in}\\
\phantom{\chi_k^{t_1}} & \chi_k^{t_3}, & \chi_k^{t_5} \end{array}
\Big| \; 1 \biggr)_{q}\\
=
\mathbb{J}_0(q, k)^2
- (2q+k^2-10k+18) \, \mathbb{J}_0(q, k) 
- 3\, \mathbb{JJ}_0(q,k) 
+4q^2+(3k^2-22k+16)q-3k^3+25k^2-40k+8.
\end{multline*}
\\ \underline{Case 10 ($t_1+t_2+t_3=t_4+t_5$):}
\begin{multline*}
q^2  \sum_{\substack{t_1,t_2,t_3,t_4=0 \\ t_1, t_2, t_3 \neq 0 \\ t_4 \neq t_1, t_2, t_3 \\ t_4 \neq t_1+t_2, t_1+t_3, t_2+t_3}}^{k-1} 
{_{3}F_2}\biggl( \begin{array}{ccc} \chi_k^{t_1}, & \chi_k^{t_2}, & \chi_k^{t_3} \vspace{.05in}\\
\phantom{\chi_k^{t_1}} & \chi_k^{t_4}, & \chi_k^{t_1+t_2+t_3-t_4} \end{array}
\Big| \; 1 \biggr)_{q}\\
=
\mathbb{J}_0(q, k)^2
- (2q+k^2-12k+24) \, \mathbb{J}_0(q, k) 
- 3\, \mathbb{JJ}_0(q,k) 
+4q^2+(3k^2-24k+22)q-3k^3+31k^2-60k+14.
\end{multline*}
The total of these ten reducible cases is
\begin{align*}
q^2 \sum_{\vec{t} \in \left({\mathbb{Z}_{k}}\right)^{5} \setminus X_k} {_{3}F_2} \left( \vec{t} \; \big| \;  1 \right)_{q,k}
&=
10 \, \mathbb{J}_0(q, k)^2
-5(3q+2k^2-12k+15) \, \mathbb{J}_0(q,k) 
-15 \, \mathbb{JJ}_0(q,k) 
\\ & \qquad
+21 q^2
+5(3k^2 - 21k+12)q
-15k^3+105k^2-135k+21.
\end{align*}
Applying (\ref{for_J0R}) and (\ref{for_JJ0S}) yields
\begin{align*}
q^2 \sum_{\vec{t} \in \left({\mathbb{Z}_{k}}\right)^{5} \setminus X_k} {_{3}F_2} \left( \vec{t} \; \big| \;  1 \right)_{q,k}
&=
10 \, \mathbb{R}_k(q)^2 + 5 \, \mathbb{R}_k(q) \left( q-2k^2+1\right) -15 \, \mathbb{S}_k(q) 
\\ & \qquad
+q^2
-5 \left(2k^2-3k+2 \right) q
+ 15k^3-10k^2+1.
\end{align*}
which completes the proof of Theorem \ref{thm_Main2}.


\section{Orbits of $X_k$ and Proofs of Corollaries \ref{cor_k2}, \ref{cor_k3}, \ref{cor_k4} and \ref{cor_3k2}-\ref{cor_3k4}}\label{sec_Orbits}
Now that we have established Theorem \ref{thm_Main2} we want to evaluate $\mathcal{K}_4(G_k(q))$ for specific $k$. A major part of this is evaluating the hypergeometric terms in 
\begin{equation}\label{for_3F2_ND}
q^2 \sum_{\vec{t} \in X_k} {_{3}F_2} \left( \vec{t} \; \big| \;  1 \right)_{q,k},
\end{equation}
where 
$X_k := \{  (t_1,t_2,t_3,t_4,t_5) \in \left({\mathbb{Z}_{k}}\right)^{5} \mid t_1,t_2,t_3 \neq 0, t_4,t_5 \, ; \, t_1+t_2+t_3 \neq t_4+t_5 \}$. 
We first note that
\begin{equation}\label{for_Xk}
|X_k| = \sum_{t_1,t_2, t_3=1}^{k-1} 
\sum_{\substack{t_4,t_5=0 \\ t_4,t_5 \neq t_1,t_2,t_3 \\ t_4+t_5 \neq t_1+t_2+t_3}}^{k-1} 1
=(k-1)(k^4-9k^3+36k^2-69k+51).
\end{equation}
Evaluating the sum in (\ref{for_Xk}) is quite a straightforward counting exercise although it does take some effort to unwind all the conditions.
Table \ref{tab_Xk} below gives the values of $|X_k|$ for small $k$.
\begin{table}[!htbp]
\centering
\begin{tabular}{| c | l |}
\hline
$k$ & $|X_k|$\\
\hline
2 & 1\\
3 & 12\\
4 & 93\\
5 & 424\\
6 & 1425\\
\hline
\end{tabular}
\vspace{6pt}
\caption{Order of $X_k$}
\label{tab_Xk}
\end{table}
As you can see, even for small $k$, there are too many terms in (\ref{for_3F2_ND}) to yield a compact formula for $K_4(G_k(q))$, if we have to evaluate them all individually. Luckily, many of the hypergeometric function summands in (\ref{for_3F2_ND}) are equal via the transformation formulae (\ref{for_3F2_T1})-(\ref{for_3F2_T7}). 

To each of the transformations (\ref{for_3F2_T1})-(\ref{for_3F2_T7}) we can associate a map on $X_k$. For example, applying (\ref{for_3F2_T1}) we get that
\begin{equation*}
{_{3}F_2}\biggl( \begin{array}{ccc} \chi_k^{t_1}, & \chi_k^{t_2}, & \chi_k^{t_3} \vspace{.05in}\\
\phantom{\chi_k^{t_1}} & \chi_k^{t_4}, & \chi_k^{t_5} \end{array}
\Big| \; 1 \biggr)_{q}
=
{_{3}F_2}\biggl( \begin{array}{ccc} \chi_k^{t_2-t_4}, & \chi_k^{t_1-t_4}, & \chi_k^{t_3-t_4} \vspace{.05in}\\
\phantom{\chi_k^{t_1}} & \chi_k^{-t_4}, & \chi_k^{t_5-t_4} \end{array}
\Big| \; 1 \biggr)_{q}.
\end{equation*}
This induces a map $T_1: X_k \to X_k$ given by 
$$T_1(t_1, t_2, t_3, t_4, t_5) = (t_2-t_4, t_1-t_4, t_3-t_4, -t_4, t_5-t_4),$$
where the addition in each component takes place in $\mathbb{Z}_5$.
Similarly, to the transformations (\ref{for_3F2_T2})-(\ref{for_3F2_T7}) we can associate the maps
\begin{align*}
T_2(t_1, t_2, t_3, t_4, t_5) &= (t_1, t_1-t_4, t_1-t_5, t_1-t_2, t_1-t_3);\\
T_3(t_1, t_2, t_3, t_4, t_5) &= (t_2-t_4, t_2, t_2-t_5, t_2-t_1, t_2-t_3);\\
T_4(t_1, t_2, t_3, t_4, t_5) &= (t_1, t_2, t_5-t_3, t_1+t_2-t_4, t_5);\\
T_5(t_1, t_2, t_3, t_4, t_5) &= (t_1, t_4-t_2, t_3, t_4, t_1+t_3-t_5);\\
T_6(t_1, t_2, t_3, t_4, t_5) &= (t_4-t_1, t_2, t_3, t_4, t_2+t_3-t_5); \, \textup{and}\\
T_7(t_1, t_2, t_3, t_4, t_5) &= (t_4-t_1, t_4-t_2, t_3, t_4, t_4+t_5-t_1-t_2)
\end{align*}
respectively. We form the group generated by $T_1, T_2, \cdots, T_7$, with operation composition of functions, and call it $\mathbb{T}_k$. We find that
\begin{align*}
\mathbb{T}_k 
&= \langle T_1, T_2, T_3, T_4, T_5, T_6, T_7 \rangle \\
&= \{T_0, T_i, T_j \circ T_{\ell}, T_4 \circ T_1, T_6 \circ T_2, T_5 \circ T_3, T_1 \circ T_4 \circ T_1 | 1 \leq i \leq 7, 1 \leq j \leq 3, 4 \leq \ell \leq 7 \},
\end{align*}
where $T_0$ is the identity map, is a group of order 24 isomorphic to the permutation group $S_4$.
$\mathbb{T}_k$ acts on $X_k$. Furthermore, the value of ${_{3}F_2} \left( \vec{t} \; \big| \;  1 \right)_{q,k}$ is constant for all $\vec{t}$ in each orbit. Therefore, we can reduce the evaluation of (\ref{for_3F2_ND}) to orbit representatives.

We can calculate explicitly the number of orbits for a given $k$, $N_k$. For $T \in \mathbb{T}_k$, let 
$X_{T} := \{\vec{t} \in X_k \mid T ({\vec{t}} \;)=\vec{t} \; \}$.
Then, by Burnside's theorem, the number of orbits is given by
\begin{equation*}
N_k = \frac{1}{24} \sum_{T \in \mathbb{T}_k} |X_{T}|.
\end{equation*}
We can evaluate each $|X_{T}|$ directly. For example,
\begin{align*}
T_1(t_1, t_2, t_3, t_4, t_5) = (t_1, t_2, t_3, t_4, t_5) 
& \Longleftrightarrow t_1=t_2, t_4=0,
\end{align*}
so
\begin{equation*}
X_{T_1} = \{ \vec{t} \in X_k \mid t_1=t_2, t_4=0 \}.
\end{equation*}
Then
\begin{align*}
|X_{T_1}| 
&= \sum_{t_1,t_3=1}^{k-1} \sum_{\substack{t_5=0 \\ t_5 \neq t_1,t_3, 2t_1+t_3}}^{k-1} 1 \\
&=  \sum_{t_1,t_3=1}^{k-1} \sum_{\substack{t_5=0 \\ t_5 \neq 2t_1+t_3}}^{k-1} 1 \; - \sum_{\substack{t_1,t_3=1 \\ t_1 \neq -t_3}}^{k-1} 1\;  - \sum_{\substack{t_1,t_3=1 \\ t_3 \neq t_1 \\2t_1 \neq 0}}^{k-1} 1 \\
&= (k-1)^3 - (k-1)(k-2) - (k-2)
\begin{cases}
(k-1) & \textup{if $k$ odd,}\\
(k-2) & \textup{if $k$ even,}
\end{cases}\\
&= 
\begin{cases}
k^3-5k^2+9k-5 & \textup{if $k$ odd,}\\
k^3-5k^2+10k-7 & \textup{if $k$ even}.
\end{cases}
\end{align*}
The other cases are similar. In summary
\begin{equation*}
|X_{T_1}| = |X_{T_7}| = |X_{T_1 \circ T_7}|
= 
\begin{cases}
k^3-5k^2+9k-5 & \textup{if $k$ odd,}\\
k^3-5k^2+10k-7 & \textup{if $k$ even},
\end{cases}
\end{equation*}
\begin{multline*}
|X_{T_2}| = |X_{T_3}| =|X_{T_4}| = |X_{T_5}| = |X_{T_6}| = |X_{T_1 \circ T_4 \circ T_1}|\\
= 
\begin{cases}
(k-1)(k-3)^2 & \textup{if $k$ odd,}\\
(k-1)(k-3)^2 +6(k-2) & \textup{if $k$ even},
\end{cases}
\end{multline*}
\begin{multline*}
|X_{T_1 \circ T_4}| = |X_{T_1 \circ T_5}| =|X_{T_1 \circ T_6}| = |X_{T_2 \circ T_7}| = |X_{T_3 \circ T_7}| = |X_{T_4 \circ T_1}|\\
= 
\begin{cases}
0 & \textup{if $k$ odd,}\\
(k-1) & \textup{if $k\equiv 2 \imod4$,}\\
(k-3) & \textup{if $k\equiv 0 \imod4$},
\end{cases}
\end{multline*}
\begin{multline*}
|X_{T_2 \circ T_4}| = |X_{T_2 \circ T_5}| =|X_{T_2 \circ T_6}| =|X_{T_3 \circ T_4}| = |X_{T_3 \circ T_5}| =|X_{T_3 \circ T_6}| = |X_{T_6 \circ T_2}| = |X_{T_5 \circ T_3}| \\
= 
\begin{cases}
k-1 & \textup{if $3 \nmid k$},\\
3(k-3) & \textup{if $3 \mid k$},
\end{cases}
\end{multline*}
and, of course,
$|X_{T_0}| = |X_k| = (k-1)(k^4-9k^3+36k^2-69k+51)$.
So, the number of orbits, by Burnside's theorem, is
\begin{equation*}
N_k = \frac{1}{24} \left[ k^5 -10k^4+54k^3-162k^2+245k-128
+
\begin{cases}
0 & \textup{if } k \equiv 1,5,7,11 \imod{12},\\
16k-64 & \textup{if } k \equiv 3,9 \imod{12},\\
45k-84 & \textup{if } k \equiv 2,10 \imod{12},\\
45k-96 & \textup{if }k \equiv 4,8 \imod{12},\\
61k-148 & \textup{if }k \equiv 6 \imod{12},\\
61k-160 & \textup{if } k \equiv 0 \imod{12}.
\end{cases}
\right].
\end{equation*} 
Table \ref{tab_Orbits} below gives the number of orbits for small $k$.
\begin{table}[!htbp]
\centering
\begin{tabular}{| c | l | l |}
\hline
$k$ & $|X_k|$ & $N_k$\\
\hline
2 & 1 & 1\\
3 & 12 & 1\\
4 & 93 & 11\\
5 & 424 & 28\\
6 & 1425 & 92\\
\hline
\end{tabular}
\vspace{6pt}
\caption{Number of Orbits}
\label{tab_Orbits}
\end{table}
So, for small $k$, the number of hypergeometric terms that need evaluating has been reduced to more manageable levels. We also note that, and we will see some evidence of this later in the proof of Corollary \ref{cor_k4}, there are other transformations which can be applied on an ad-hoc basis to reduce these numbers even further. We now prove Corollaries \ref{cor_k2}, \ref{cor_k3} and \ref{cor_k4}.

\begin{proof}[Proof of Corollary \ref{cor_k2}]
We apply Theorem \ref{thm_Main2} with $k=2$. Now $\mathbb{R}_2(q) = \mathbb{S}_2(q) = 0$ as there are no indices that satisfy the conditions of the sum in each case. Also, 
$\vec{t} = (1,1,1,0,0)$ is the only element of $X_2$. Therefore
\begin{equation}\label{for_k4G2}
\mathcal{K}_4(G_2(q)) =
\frac{q(q-1)}{2^9 \cdot 3} 
\Biggl[ 
 q^2 -20 q+ 81
+q^2 
{_{3}F_2} {\left( \begin{array}{ccc} \varphi, & \varphi, & \varphi \\
\phantom{\varphi} & \varepsilon, & \varepsilon \end{array}
\Big| \; 1 \right)}_{q}
\Biggr].
\end{equation}
Combining \cite[Thm 4.37]{G2} and Proposition \ref{prop_JacXfer} we get that
\begin{align*}
q^2 {_{3}F_2} {\left( \begin{array}{ccc} \varphi, & \varphi, & \varphi \\
\phantom{\varphi} & \varepsilon, & \varepsilon \end{array}
\Big| \; 1 \right)}_{q}
&=q^2 \left[\binom{\chi_4}{\varphi} \binom{\chi_4}{\bar{\chi_4}} +\binom{\bar{\chi_4}}{\varphi} \binom{\bar{\chi_4}}{\chi_4} \right]\\
&= J(\chi_4,\chi_4)^2 + J(\bar{\chi_4},\bar{\chi_4})^2, 
\end{align*}
where $\chi_4$ is a character of order four of $\mathbb{F}_{q}^*$.
So, by Lemma \ref{lem_JacOrder4}(2),
\begin{equation}\label{for_3F2_qOno}
q^2 {_{3}F_2} {\left( \begin{array}{ccc} \varphi, & \varphi, & \varphi \\
\phantom{\varphi} & \varepsilon, & \varepsilon \end{array}
\Big| \; 1 \right)}_{q}
= 2x^2-2y^2 = 4x^2-2q = 2q-4y^2,
\end{equation}
which is generalization of \cite[Thm 4]{O} to prime powers.
Substituting (\ref{for_3F2_qOno}) into (\ref{for_k4G2}) yields the result.
\end{proof}

\begin{proof}[Proof of Corollary \ref{cor_k3}]
We apply Theorem \ref{thm_Main2} with $k=3$. Note $\mathbb{S}_3(q) = 0$ as there are no indices that satisfy the conditions of the sum. Also, $X_3$ contains only one orbit of order 12 with representative $\vec{t} = (1,1,2,0,0)$. Therefore
\begin{multline}\label{for_k4G3}
\mathcal{K}_4(G_3(q)) =
\frac{q(q-1)}{2^3 \cdot 3^7} 
\Biggl[ 
10 \, \mathbb{R}_3(q)^2 + 5 \left( q-17\right) \mathbb{R}_3(q) +q^2\\
-55 q + 316
+12 \, q^2 \,
{_{3}F_2} {\left( \begin{array}{cccc} \chi_3, & \chi_3, &  \bar{\chi_3} \\
\phantom{\chi_3} & \varepsilon, &  \varepsilon \end{array}
\Big| \; 1 \right)}_{q}
\Biggr].
\end{multline}
Now $\mathbb{R}_3(q) = J(\chi_3,\chi_3) + J(\bar{\chi_3},\bar{\chi_3}) = c$ by Lemma \ref{lem_JacOrder3}. Substituting into (\ref{for_k4G3}) yields the result.
\end{proof}

\begin{proof}[Proof of Corollary \ref{cor_k4}]
Using Proposition \ref{prop_JacXfer} and Lemma \ref{lem_JacOrder4}(1) we get that 
\begin{equation}\label{for_R4}
\mathbb{R}_4(q) = 3(J(\chi_4,\chi_4) + J(\bar{\chi_4},\bar{\chi_4})) =-6x.
\end{equation}
Using Propositions \ref{prop_JacXfer} \& \ref{prop_JacProdq} and Lemma \ref{lem_JacOrder4}(2) we have
\begin{equation}\label{for_S4}
\mathbb{S}_4(q) = 4q + J(\chi_4,\chi_4)^2 + J(\bar{\chi_4},\bar{\chi_4})^2 = 4x^2+2q.
\end{equation}
$X_4$ contains eleven orbits with representatives 
$(1,1,1,0,0)^6$,
$(3,3,3,0,0)^6$,
$(1,3,3,2,0)^4$,
$(3,1,1,2,0)^4$,
$(2,1,3,0,0)^{12}$,
$(1,3,2,0,0)^6$,
$(2,3,1,0,0)^{12}$,
$(1,2,2,0,0)^{24}$,
$(2,2,1,0,0)^6$,
$(1,1,3,0,0)^{12}$, and
$(2,2,2,0,0)^1$,
where the superscripts represent the order of the orbit. Some of the corresponding hypergeometric functions are equal and some can be reduced. In particular, by definition,
\begin{equation}\label{for_k4Eq1}
{_{3}F_2} {\left( \begin{array}{ccc} \varphi, & \chi_4, & \bar{\chi_4}  \\
\phantom{\varphi} & \varepsilon, & \varepsilon \end{array}
\Big| \; 1 \right)}_{q}\\
=
{_{3}F_2} {\left( \begin{array}{cccc} \chi_4, & \bar{\chi_4}, &  \varphi \\
\phantom{\chi_4} & \varepsilon, &  \varepsilon \end{array}
\Big| \; 1 \right)}_{q}
=
{_{3}F_2} {\left( \begin{array}{cccc} \varphi, & \bar{\chi_4}, &  \chi_4 \\
\phantom{\varphi} & \varepsilon, &  \varepsilon \end{array}
\Big| \; 1 \right)}_{q},
\end{equation}
and
\begin{equation}\label{for_k4Eq2}
{_{3}F_2} {\left( \begin{array}{cccc} \chi_4, & \varphi, &  \varphi, \\
\phantom{\chi_4} & \varepsilon, &  \varepsilon \end{array}
\Big| \; 1 \right)}_{q}
=
{_{3}F_2} {\left( \begin{array}{cccc}\varphi, & \varphi, &  \chi_4 \\
\phantom{\chi_4} & \varepsilon, &  \varepsilon \end{array}
\Big| \; 1 \right)}_{q}.
\end{equation}

By \cite[Thm 4.37]{G2}, Proposition \ref{prop_JacXfer} and Lemma \ref{lem_JacOrder8}
\begin{align}\label{for_k4Red2}
\notag
q^2 & {_{3}F_2} {\left( \begin{array}{ccc} \chi_4, & \chi_4, & \chi_4 \\
\phantom{\chi_4} & \varepsilon, & \varepsilon \end{array}
\Big| \; 1 \right)}_{q}
+q^2 {_{3}F_2} {\left( \begin{array}{ccc} \bar{\chi_4}, & \bar{\chi_4}, & \bar{\chi_4} \\
\phantom{\bar{\chi_4}} & \varepsilon, & \varepsilon \end{array}
\Big| \; 1 \right)}_{q}\\
\notag
&=q^2 \left[\binom{\chi_8}{\chi_8^2} \binom{\chi_8}{\chi_8^3} +\binom{\chi_8^5}{\chi_8^2} \binom{\chi_8^5}{\chi_8^7}
+\binom{\chi_8^7}{\chi_8^6} \binom{\chi_8^7}{\chi_8^5} +\binom{\chi_8^3}{\chi_8^6} \binom{\chi_8^3}{\chi_8} \right]\\
\notag
&= \chi_8(-1) \left[ J(\chi_8,\chi_8) J(\chi_8,\chi_8^2) + J(\chi_8^5,\chi_8^5) J(\chi_8,\chi_8^2) 
\right. \\ \notag & \qquad \qquad \qquad \qquad \quad  \left. 
+ J(\chi_8^7,\chi_8^7) J(\chi_8^7,\chi_8^6) + J(\chi_8^3,\chi_8^3) J(\chi_8^7,\chi_8^6)  \right]\\
\notag
&= 2 \, \chi_8(-1)  \, Re \left[J(\chi_8,\chi_8) J(\chi_8,\chi_8^2) + J(\chi_8^3,\chi_8^3) J(\chi_8^7,\chi_8^6) \right]\\
\notag
&= 4 \, \chi_8(-1)  \, Re(J(\chi_8,\chi_8)) \, Re(J(\chi_8,\chi_8^2))\\
 &= - 4  u x.
\end{align}
Similarly
\begin{align}\label{for_k4Red2}
q^2 & {_{3}F_2} {\left( \begin{array}{ccc} \chi_4, & \bar{\chi_4}, & \bar{\chi_4} \\
\phantom{\chi_4} & \varphi, & \varepsilon \end{array}
\Big| \; 1 \right)}_{q}
+q^2 {_{3}F_2} {\left( \begin{array}{ccc} \bar{\chi_4}, & \chi_4, & \chi_4 \\
\phantom{\bar{\chi_4}} & \varphi, & \varepsilon \end{array}
\Big| \; 1 \right)}_{q}
= - 4  u x.
\end{align}
By \cite[Thm 4.38(ii)]{G2}, Proposition \ref{prop_JacXfer} and Lemma \ref{lem_JacOrder8}
\begin{align}\label{for_k4Red3}
\notag
q^2  {_{3}F_2} {\left( \begin{array}{ccc} \varphi, & \chi_4, & \bar{\chi_4}  \\
\phantom{\varphi} & \varepsilon, & \varepsilon \end{array}
\Big| \; 1 \right)}_{q}
&=q^2 \left[\binom{\chi_8}{\chi_8^4} \binom{\chi_8}{\chi_8^5} +\binom{\chi_8^5}{\chi_8^4} \binom{\chi_8^5}{\chi_8} \right]\\
\notag
&= J(\chi_8,\chi_8^3)^2 +  J(\chi_8^5,\chi_8^7)^2 \\
\notag
&= (u+v \sqrt{-2})^2 + (u-v\sqrt{-2})^2\\
&=4u^2-2q.
\end{align}
Taking $k=4$ in Theorem \ref{thm_Main2} and accounting for (\ref{for_3F2_qOno}), (\ref{for_R4})-(\ref{for_k4Red3}) yields the result.
\end{proof}

We now also have all the ingredients to prove Corollaries \ref{cor_3k2}-\ref{cor_3k4}.

\begin{proof}[Proof of Corollaries \ref{cor_3k2}-\ref{cor_3k4}]
We've seen above that $\mathbb{R}_2(q) = 0$, $\mathbb{R}_3(q) = c$ and $\mathbb{R}_4(q) =-6x$. Taking $k=2,3,4$ in Theorem \ref{thm_Main3} yields the results.
\end{proof}


\section{Lower bounds for the multicolor Ramsey numbers}\label{sec_Ramsey}
Let $k\geq 2$ and $n_1,n_2, \cdots, n_k$ be positive integers. Let $K_l$ denote the complete graph of order $l$. The multicolor Ramsey number $R(n_1,n_2, \cdots, n_k)$ is the smallest integer $l$ satisfying the property that, if the edges of $K_l$ are colored in $k$ colors, then for some $1 \leq i \leq k$, $K_l$ contains a complete subgraph $K_{n_i}$ in color $i$. If $n_1=n_2=\cdots = n_k $ then we use the abbreviated $R_k(n_1)$ to denote the multicolor Ramsey number.

Recall the generalized Paley graph of order $q$, $G_k(q)$, for $q$ a prime power such that $q \equiv 1 \imod {k}$ if $q$ is even, or, $q \equiv 1 \imod {2k}$ if $q$ is odd. $S_k$ is the subgroup of the multiplicative group $\mathbb{F}_q^{\ast}$ of order $\frac{q-1}{k}$ containing the $k$-th power residues, i.e., if $\omega$ is a primitive element of $\mathbb{F}_q$, then $S_k = \langle \omega^k \rangle$. Then $G_k(q)$ is the graph with vertex set $\mathbb{F}_q$ where $ab$ is an edge if and only if $a-b \in S_k$.

We now define subsets of $\mathbb{F}_q^{\ast}$, $S_{k,i} := \omega^{i} S_k$, for $0 \leq i \leq k-1$, and the related graphs $G_{k,i}(q)$ with vertex set $\mathbb{F}_q$ where $ab$ is an edge if and only if $a-b \in S_{k,i}$. Each $G_{k,i}(q)$ is isomorphic to $G_{k,0}(q)=G_k(q)$, the generalized Paley graph, via the map $f:V(G_k(q)) \to V(G_{k,i}(q))$ given by $f(a)=\omega^i a$. Now consider the complete graph $K_q$ whose vertex set is taken to be $\mathbb{F}_q$ and whose edges are colored in $k$ colors according to $ab$ has color $i$ if $a-b \in S_{k,i}$. Note that the color $i$ subgraph of $K_q$ is $G_{k,i}(q)$. Thus, this $K_q$ has a subgraph $K_l$ in a single color if and only if the generalized Paley graph contains a subgraph $K_l$.  Therefore, if $\mathcal{K}_l(G_k(q_1))=0$ for some $q_1$, then $q_1<R_k(l)$. 

We start with the ${l=4}$ case. For a given $k$, we search for the greatest $q$ such that $\mathcal{K}_4(G_k(q))=0$, thus establishing that $q<R_k(4)$.
When $k=3$ we search all $q \leq 230$, which is a known upper bound for $R_3(4)$ \cite {R}, and use Corollary \ref{cor_k3} to confirm the value of $\mathcal{K}_4(G_k(q))$. When $q=127$, $c=-20$ and 
$${_{3}F_2} {\left( \begin{array}{cccc} \chi_3, & \chi_3, &  \bar{\chi_3} \\
\phantom{\chi_3} & \varepsilon, &  \varepsilon \end{array}
\Big| \; 1 \right)}_{127}
=-205.$$
Thus $\mathcal{K}_4(G_3(127))=0$ and so $128 \leq R(4,4,4)$, which proves Corollary \ref{cor_Ramseyk3}.
When $k=4$ we search all $q \leq 6306$, which is a known upper bound for $R_4(4)$ \cite {R}. We apply Corollary \ref{cor_k4} with $q=457$. In this case $x=21$, $u=-13$, 
$${_{3}F_2} {\left( \begin{array}{cccc} \chi_4, & \chi_4, &  \bar{\chi_4} \\
\phantom{\chi_3} & \varepsilon, &  \varepsilon \end{array}
\Big| \; 1 \right)}_{457}
=290,
\qquad \textup{and} \qquad
{_{3}F_2} {\left( \begin{array}{cccc} \varphi, & \varphi, &  \chi_4 \\
\phantom{\chi_3} & \varepsilon, &  \varepsilon \end{array}
\Big| \; 1 \right)}_{457}
=-590.$$
Thus $\mathcal{K}_4(G_4(457))=0$ and so $458 \leq R(4,4,4,4)$, which proves Corollary \ref{cor_Ramseyk4}.
Similar searches for $k=5,6$, using Theorem \ref{thm_Main2}, yield $942 \leq R_5(4)$ and $3458 \leq R_6(4)$ which fall well short of known bounds \cite{R, XZER}.

When ${l=3}$, Corollary \ref{cor_Ramsey3k234} follows easily from Corollaries \ref{cor_3k2}-\ref{cor_3k4}. We can use Theorem \ref{thm_Main3} with $k=5,6$ to get $102 \leq R_5(3)$ and $278 \leq R_6(3)$, which again fall well short of known bounds. In fact it is known that $162 \leq R_5(3) \leq 307$ and $538 \leq R_6(3) \leq 1838$ \cite{XZER}.


\section{Connections to Modular Forms}\label{sec_ModularForms}
When $q=p$ is prime, many of the quantities in Corollaries \ref{cor_k2}, \ref{cor_k3} and \ref{cor_k4} can be related to the $p$-th Fourier coefficients of certain modular forms. Let $x,y,c,d,u,v$ be as defined in those corollaries.  

Consider the unique newform $f \in S_3(\Gamma_0(16), (\tfrac{-4}{\cdot}))$, which has complex multiplication (CM) by $\mathbb{Q}(i)$, where
$$
f(z) = \sum_{n=1}^{\infty} \alpha(n) q^n = q\prod_{m=1}^{\infty}(1-q^{4m})^6,  \quad q := e^{2\pi i z}.
$$
Then \cite{MOS, O2, Z}, for $p \equiv 1 \imod4$,
$\alpha(p) = 2x^2-2y^2 = 2p-4y^2$,
and Corollary \ref{cor_k2} yields
\begin{cor}\label{cor_k2_MF}
Let $p \equiv 1 \pmod {4}$ be prime. Then
\begin{equation*}
\mathcal{K}_4(G(p)) =
\frac{p(p-1)((p-9)^2 - 2p + \alpha(p))}{2^9 \cdot 3}. 
\end{equation*}
\end{cor}

Consider the unique newform $g_1 \in S_2(\Gamma_0(27))$, which has CM by $\mathbb{Q}(\sqrt{-3})$, where
$$
g_1(z) = \sum_{n=1}^{\infty} \beta_1(n) q^n = q\prod_{m=1}^{\infty}(1-q^{3m})^2 (1-q^{9m})^2.
$$
Then \cite{MO}, for $p \equiv 1 \imod6$, $\beta_1(p) = -c$.
Also consider the non-CM newform $g_2 \in S_3(\Gamma_0(27), (\tfrac{-3}{\cdot}))$ where 
$$
g_2(z) = \sum_{n=1}^{\infty} \beta_2(n) q^n = q + 3 i q^{2} -5 q^{4} -3 i q^{5} + 5 q^{7} -3 i q^{8} + \cdots .
$$
Then numerical evidence \cite{LMFDB} suggests that, for $p \equiv 1 \imod6$,
$$p^2 
{_{3}F_2} {\left( \begin{array}{cccc} \chi_3, & \chi_3, &  \bar{\chi_3} \\
\phantom{\chi_3} & \varepsilon, &  \varepsilon \end{array}
\Big| \; 1 \right)}_{p} 
\stackrel{?}{=} \beta_2(p).
$$
If so, then Corollary \ref{cor_k3} yields, for $p \equiv 1 \imod6$,  
\begin{equation*}
\mathcal{K}_4(G_3(p)) \stackrel{?}{=} 
\frac{p(p-1)}{2^3 \cdot 3^7}. 
\Biggl[ 
p^2-5p(\beta_1(p)+11)+10 \, \beta_1(p)^2+85 \, \beta_1(p)+316 +12 \, \beta_2(p)
\Biggr].
\end{equation*}

In the $k=4$ case, we consider the following newforms:
\begin{align*}
h_1(z) &= \sum_{n=1}^{\infty} \gamma_1(n) q^n = q\prod_{m=1}^{\infty} (1-q^{4m})^{-2}  (1-q^{8m})^8 (1-q^{16m})^{-2} \in S_2(\Gamma_0(64));\\
h_2(z) &= \sum_{n=1}^{\infty} \gamma_2(n) q^n = q + 2 i q^{3} - q^{9} -6 i q^{11} -6 q^{17} + 2 i q^{19} + 5 q^{25}  \cdots \in  S_2(\Gamma_0(64), \Psi_1);\\
h_3(z) &= \sum_{n=1}^{\infty} \gamma_3(n) q^n = q\prod_{m=1}^{\infty} (1-q^{m})^2 (1-q^{2m}) (1-q^{4m}) (1-q^{8m})^2 \in S_3(\Gamma_0(8),(\tfrac{-2}{\cdot})); \textup{and}\\
h_4(z) &= \sum_{n=1}^{\infty} \gamma_4(n) q^n = q + 4iq^3 +2 q^5 -8iq^7 -7q^9 - 4iq^{11} -14 q^{13} + \cdots \in S_3(\Gamma_0(32), (\tfrac{-4}{\cdot})),
\end{align*}
where $\Psi_1$ is the Dirichlet character modulo 64 sending generators $(63,5)\to(1,-1)$, $h_1$ has CM by $\mathbb{Q}(i)$, and, $h_2$ and $h_3$ have CM by $\mathbb{Q}(\sqrt{-2})$.

When $p \equiv 1 \imod4$, $\gamma_1(p) = 2x$ \cite{MO, Z}. When $p \equiv 1 \imod 8$, $\gamma_3(p) = 2u^2-4v^2 = 4u^2-2p$ \cite{GMY}. Further, numerical evidence \cite{LMFDB} suggests that $\gamma_2(p) \stackrel{?}{=}  -2u$, when $p \equiv 1 \imod8$. It should be a relatively straightforward exercise to establish this relation using Hecke characters and a similar construction to that in \cite{GMY}. 
If $p \equiv 1 \pmod 4$, then \cite{MP}
\begin{equation*}
p^2 \, {_{3}F_2} \biggl( \begin{array}{ccc} \chi_4, & \varphi, & \varphi \vspace{.02in}\\
\phantom{\chi_4} & \varepsilon, & \varepsilon \end{array}
\Big| \; 1 \biggr)_{p}
=\gamma_4(p).
\end{equation*}
If we let 
\begin{equation*}
p^2 \,{_{3}F_2} \biggl( \begin{array}{ccc} \chi_4, & \chi_4 & \bar{\chi_4} \vspace{.02in}\\
\phantom{\chi_4} & \varepsilon, & \varepsilon \end{array}
\Big| \; 1 \biggr)_{p}
=\gamma_5(p),
\end{equation*}
then it looks like $\gamma_5(p)$ is a twist of the $p$-th Fourier coefficient of a newform 
$h_5(z) \in S_3(\Gamma_0(128), \Psi_2),$
where $\Psi_2$ is the Dirichlet character modulo 128 sending generators $(127,5)\to(-1,-1)$.
So, when $p \equiv 1 \imod8$, we would expect
\begin{multline*}
\mathcal{K}_4(G_4(p)) \stackrel{?}{=}
\frac{p(p-1)}{2^{15} \cdot 3}. 
\Biggl[ 
p^2-p(15 \, \gamma_1(p)+142)+76 \, \gamma_1(p)^2+ 465 \, \gamma_1(p)+801 \\
+10 \, \gamma_1(p) \gamma_2(p)
+30 \, \gamma_3(p)
+30 \, \gamma_4(p)
+12 \, \gamma_5(p)
\Biggr].
\end{multline*}

In general, establishing relations between finite field hypergeometric functions and coefficients of non-CM modular forms is not a straightforward exercise. The main method used in the known cases is to apply the Eichler-Selberg trace formula for Hecke operators to isolate the Fourier coefficients of the form, and then connect these traces to hypergeometric values by counting isomorphism classes of elliptic curves with prescribed torsion  (see for example \cite{A, AO, FOP, F, L, MP}). This is a long and tedious process and, to date, has been carried out on an ad-hoc basis in each case to accommodate the specifics of each form. Relations with CM modular forms, however, are generally much easier to establish. These usually reduce to evaluating Jacobi sums, as we have seen above (see also \cite{M}).

We can also express Corollaries \ref{cor_3k3} and \ref{cor_3k4}, when $q$ is prime, in terms of coefficients of modular forms using the connections outlined above.
\begin{cor}\label{cor_3k3_MF}
Let $p\equiv 1 \pmod {6}$ be prime. Then
\begin{equation*}
\mathcal{K}_3(G_3(q)) =
\frac{q(q-1)(q- \beta_1(p)-8)}{2 \cdot 3^4}.
\end{equation*}
\end{cor}

\begin{cor}\label{cor_3k4_MF}
Let $p \equiv 1 \pmod {8}$ be prime.
Then
\begin{equation*}
\mathcal{K}_3(G_4(q)) =
\frac{q(q-1)(q- 3 \, \gamma_1(p)-11)}{2^{7} \cdot 3}. 
\end{equation*}
\end{cor}



\vspace{6pt}

\begin{thebibliography}{999}

\bibitem{A}
S. Ahlgren, \emph{The points of a certain fivefold over finite fields and the twelfth power of the eta function}, Finite Fields Appl. 8 (2002), no. 1, 18--33.
 
\bibitem{AO} 
S. Ahlgren, K. Ono, \emph{Modularity of a certain Calabi-Yau threefold}, Monatsh. Math. \textbf{129} (2000), no.~3, 177--190.

\bibitem{BEW} 
B. Berndt, R. Evans, K. Williams, \emph{Gauss and Jacobi Sums}, Canadian Mathematical Society Series of Monographs and Advanced Texts, A Wiley-Interscience Publication, John Wiley \& Sons, Inc., New York, 1998.

\bibitem{C}
F.R.K. Chung, \emph{On the Ramsey numbers N(3,3,…,3;2)}, Discrete Math. 5 (1973), 317--321.
 
\bibitem{EPS}
R.J. Evans, J.R. Pulham, J. Sheehan, \emph{On the number of complete subgraphs contained in certain graphs}, J. Combin. Theory Ser. B \textbf{30} (1981), no. 3, 364--371. 

\bibitem{FOP} 
S. Frechette, K. Ono, M. Papanikolas, \emph{Gaussian hypergeometric functions and traces of Hecke operators}, Int. Math. Res. Not. \textbf{2004}, no. 60, 3233--3262.

\bibitem{FKR}
S. Fettes, R. Kramer, S. P. Radziszowski, \emph{An upper bound of 62 on the classical Ramsey number R(3,3,3,3)}, Ars Combin. 72 (2004), 41--63.

\bibitem{F} 
J. G. Fuselier, \emph{Traces of Hecke operators in level 1 and Gaussian hypergeometric functions}, Proc. Amer. Math. Soc. 141 (2013), no. 6, 1871--1881.

\bibitem{GMY}
A. Gomez, D. McCarthy, D. Young, \emph{Apéry-like numbers and families of newforms with complex multiplication}, Res. Number Theory 5 (2019), no. 1, Paper No. 5, 12 pp.

\bibitem{G} 
J. Greene, \emph{Character sum analogues for hypergeometric and generalized hypergeometric functions over finite fields}, Thesis (Ph.D.) University of Minnesota, 1984.

\bibitem{G2} 
J. Greene, \emph{Hypergeometric functions over finite fields}, Trans. Amer. Math. Soc. {\bf 301} (1987), no. 1, 77--101.

\bibitem{GG}
R.E. Greenwood, A.M. Gleason, \emph{Combinatorial relations and chromatic graphs}, Canadian J. Math. 7 (1955), 1--7. 

\bibitem{HI}
R. Hill, R.W. Irving, \emph{On Group Partitions  Associated  with Lower Bounds for Symmetric Ramsey Numbers}, European J. Combin. 3 (1982) 35--50.


\bibitem{J}
G. Jones, \emph{Paley and the Paley graphs}, arXiv:1702.00285v1.

\bibitem{L} C. Lennon, \emph{Trace formulas for Hecke operators, Gaussian hypergeometric functions, and the modularity of a threefold}, J. Number Theory \textbf{131} (2011), no.~12, 2320--2351.

\bibitem{LP}
T.K. Lim, C. Praeger, \emph{On generalized Paley graphs and their automorphism groups}, Michigan Math. J. 58 (2009), no. 1, 293--308. 

\bibitem{LMFDB} LMFDB - The L-functions and Modular Forms Database, www.lmfdb.org. 

\bibitem{MO}
Y. Martin, K. Ono, \emph{Eta-quotients and elliptic curves}, Proc. Amer. Math. Soc. 125 (1997), no. 11, 3169--3176.

\bibitem{McC11} 
D. McCarthy, \emph{Multiplicative relations for Fourier coefficients of degree 2 Siegel eigenforms}, J. Number Theory \textbf{170} (2017), 263--281.

\bibitem{MP}
D. McCarthy, M. Papanikolas, \emph{A finite field hypergeometric function associated to eigenvalues of a Siegel eigenform}, Int. J. Number Theory 11 (2015), no. 8, 2431--2450.
 
\bibitem{MOS}
D. McCarthy, R. Osburn, A. Straub, \emph{Sequences, modular forms and cellular integrals},  Math. Proc. Cambridge Philos. Soc. 168 (2020), no. 2, 379--404.

\bibitem{M} 
E. Mortenson, \emph{Supercongruences for truncated ${}_{n+1}F_{n}$ hypergeometric series with applications to certain weight three newforms}, Proc. Amer. Math. Soc. \textbf{133} (2005), no.~2, 321--330.

\bibitem{O2} 
K. Ono, \emph{A note on the Shimura correspondence and the Ramanujan $\tau(n)$ function}, Utilitas Math. 47 (1995), 153--160.

\bibitem{O} 
K. Ono, \emph{Values of Gaussian hypergeometric series}, Trans. Amer. Math. Soc. \textbf{350} (1998), no.~3, 1205--1223.

\bibitem{P}
R. E. A. C. Paley, \emph{On orthogonal matrices}, J. Math. Phys. 12 (1933), 311--320.

\bibitem{PAR}
J. C. Parnami, M. K. Agrawal, A. R. Rajwade, \emph{Jacobi sums and cyclotomic numbers for a finite field}, Acta Arith. 41 (1982), no. 1, 1--13. 

\bibitem{KR}
S. A. Katre, A. R. Rajwade, \emph{Resolution of the sign ambiguity in the determination of the cyclotomic numbers of order 4 and the corresponding Jacobsthal sum}, Math. Scand. 60 (1987), no. 1, 52--62. 

\bibitem{R}
S. P. Radziszowski, \emph{Small Ramsey numbers}, Electron. J. Combin. 1 (1994), Dynamic Survey 1, 30 pp.

\bibitem{XZER}
X. Xiaodong, X. Zheng, G. Exoo, S. P. Radziszowski, \emph{Constructive lower bounds on classical multicolor Ramsey numbers}, Electron. J. Combin. 11 (2004), no. 1, Research Paper 35, 24 pp.

\bibitem{Z}
D. Zagier, \emph{Elliptic modular forms and their applications}, The 1-2-3 of modular forms, 1--103, Universitext, Springer, Berlin, 2008.

\end{thebibliography}
\end{document}